\documentclass[11pt,reqno]{amsart}
\usepackage{graphicx}
\usepackage{amssymb}
\usepackage[usenames,dvipsnames]{color}
\newtheorem{theorem}{Theorem}[section]
\newtheorem{lemma}{Lemma}[section]
\newtheorem{proposition}{Proposition}[section]
\newtheorem{definition}{Definition}[section]

\newtheorem{remark}{Remark}[section]

\numberwithin{equation}{section}

\def\essinf{\mathop{\rm ess~inf}}
\def\R{\mathbb{R}}
\def\Z{\mathbb{Z}}
\def\pa{\partial}

\begin{document}

\title[Transition fronts of two species competition lattice systems in random media]{Transition fronts of two species competition lattice systems in random media}

\author{Feng Cao}
\address{Department of Mathematics, Nanjing University of Aeronautics and Astronautics, Nanjing, Jiangsu 210016, P. R. China}
\email{fcao@nuaa.edu.cn}
\thanks{Research of F. Cao was supported by NSF of China No. 11871273, and the Fundamental Research Funds for the Central Universities No. NS2018047.}

\author{Lu Gao}
\address{Department of Mathematics, Nanjing University of Aeronautics and Astronautics, Nanjing, Jiangsu 210016, P. R. China}
\email{gaolunuaa@163.com}





\begin{abstract}
The current paper is devoted to the study of existence and non-existence of transition fronts for two species competition lattice system in random media, and explore
the influence of randomness of the media on the wave profiles and wave speeds of such transition fronts. We first establish comparison principle for sub-solutions and super-solutions of the related cooperative system. Next, under some proper assumptions, we construct appropriate sub-solutions and super-solutions for the cooperative system. Finally, we show that random transition fronts exist if their least mean speed is greater than an explicit threshold and there is no random transition front with least mean speed less than the threshold.
\end{abstract}

\maketitle



\section{Introduction}

The current paper studies the existence of transition fronts of the following two species competition lattice random system
\begin{equation}\label{main-eqn}
\left\{
\begin{aligned}
\dot{u}_i(t) &=u_{i+1}( t ) -2u_i( t) +u_{i-1}( t) +u_i( t)(a_1(\theta _t\omega)- b_1(\theta _t\omega)u_i( t)-c_1(\theta _t\omega)v_i( t)) ,\\
\dot{v}_i(t) &=v_{i+1}( t ) -2v_i( t) +v_{i-1}( t) +v_i( t)(a_2(\theta _t\omega)- b_2(\theta _t\omega)u_i( t)-c_2(\theta _t\omega)v_i( t)),
\end{aligned}
\right.
\end{equation}
where $ i\in \mathbb{Z}$, $t\in\mathbb{R}$, $\omega\in\Omega$, $(\Omega,\mathcal{F},\mathbb{P})$ is a given probability space, $\theta_t$ is an ergodic metric dynamical system on $\Omega$, $a_i( \cdot ) :\Omega\rightarrow\mathbb{R}$, $b_i( \cdot ) :\Omega\rightarrow(0,\infty)$, $c_i( \cdot ) :\Omega\rightarrow(0,\infty)$ $(i=1,2)$ are measurable, and for every $\omega\in\Omega$, $a^\omega_i(t):=a_i(\theta_t\omega)$, $b^\omega_i(t):=b_i(\theta_t\omega)$, $c^\omega_i(t):=c_i(\theta_t\omega)$ $(i=1,2)$ are locally H\"older continuous in $t\in\R$. Moreover, we assume $b_i(\theta_t\omega)>0$, $c_i(\theta_t\omega)>0$ $(i=1,2)$ for every $\omega\in\Omega$ and $t\in\R$.

System \eqref{main-eqn} is a spatial-discrete counterpart of the following two species competition system with random dispersal,
\begin{equation}\label{continuous-main-eqn} 
\left\{
\begin{aligned}
\partial_tu &=u_{xx}+u(a_1(\theta _t\omega)- b_1(\theta _t\omega)u-c_1(\theta _t\omega)v) ,\\
\partial_tv &=v_{xx}+v(a_2(\theta _t\omega)- b_2(\theta _t\omega)u-c_2(\theta _t\omega)v),
\end{aligned}
\right.
\end{equation}
Systems \eqref{main-eqn} and \eqref{continuous-main-eqn} are widely used to model the population dynamics of competitive species when the movement
or internal dispersal of the organisms occurs between non-adjacent and adjacent locations, respectively (see, for example, \cite{Ch03,MaPa03,ShKa97, ShSw90}). Note that system  \eqref{continuous-main-eqn} often models the evolution of population densities of competitive species in which the
internal interaction or movement of the organisms occurs randomly between adjacent
spatial locations and is described by the differential operator, referred to as the {\it random dispersal operator}. System \eqref{main-eqn} arises in modeling the evolution
of population densities of competitive species in which the internal interaction or
movement of the organisms occurs between non-adjacent spatial locations and is
described by the difference operator, referred to as the {\it discrete dispersal operator}.

In \eqref{main-eqn} and \eqref{continuous-main-eqn}, the functions $a_1$, $a_2$ represent the respective growth rates of the two species, $b_1$, $c_2$ account for self-regulation of the respective species, and $c_1$, $b_2$ account for competition between the two species. Two of the central dynamical issues about \eqref{main-eqn} and \eqref{continuous-main-eqn} are spatial spreading speeds
and traveling wave solutions.  A huge amount of research has been carried out toward the spatial spreading speeds and traveling wave solutions of system \eqref{continuous-main-eqn} in spatially and temporally homogeneous media (see, for example, \cite{CoGa84,Du83,GuLi11,Ho98,Hu10,Ka95,Ka97,LeLiWe02,LiWeLe05,LiLiRu06,LiZh07,LiZh10,WeLeLi02}) or spatially and/or temporally periodic media (see, for example, \cite{BaWa13,FaYuZh17,YuZh17,ZhRu11}). Recently, Bao, Li, Shen and Wang in \cite{BaLiShWa18} studied the spatial spreading speeds and linear determinacy of diffusive cooperative/competitive system in time recurrent environments. Bao in \cite{Ba18} studied the spatial spreading speeds and generalized traveling waves of competition system in general time heterogeneous media. 

As for the lattice system arising in competition models, to the best of our knowledge, there are only a few works on the related topics. The reader is referred to \cite{GuWu11,GuWu12} for the  study on the spatial spreading speeds and traveling wave solutions for competition lattice system in time independent habitats.
We note that Cao and Gao in \cite{CaGa19} studied the existence and stability of random transition fronts for KPP-type one species lattice random equations.  The reader is referred to \cite{CaSh17,CaoShen17,GuHa06,MaZh08,Sh09,WuZo97,ZiHaHu93} for the study on the spatial spreading speeds and traveling wave solutions for KPP-type one species lattice equations in homogeneous or periodic or time heterogeneous media.

The purpose of our current paper is to study the traveling wave solutions of two species competition lattice system with general time dependence. Since in nature, many systems are subject to irregular influences arisen from various kind of noise, it is of great importance to take the randomness of the environment into account and study the existence and non-existence of random transition fronts of competition lattice system in random media. Due to the lack of space regularity, we need finding new approach to get the existence of transition fronts when dealing with spatial-discrete system \eqref{main-eqn}. We point out that the method used here can also be used to get the existence and non-existence of transition fronts for two species competition lattice system in general time dependent habitats.

Let
$$l^\infty(\Z)=\{u=\{u_i\}_{i\in \Z}:\sup \limits_{i \in\Z}|u_i|<\infty\}$$
with norm $\|u\|=\|u\|_\infty=\sup_{i\in\Z}|u_i|$, and
$$l^{\infty,+}(\Z)=\{u\in l^\infty(\Z): \inf\limits_{i\in\Z}u_i\geq 0\}.$$
For $u,v\in l^\infty(\Z)$, we define
$$
u\ge v \quad {\rm if}\quad u-v\in l^{\infty,+}(\Z).
$$
Then for any given $(u_0, v_0 )\in l^\infty(\Z)\times l^\infty(\Z)$ , \eqref{main-eqn} has a unique (local) solution
$(u(t;u_0,v_0,\omega)$, $v(t;u_0,v_0,\omega))$ $=\{(u_i(t;u_0,v_0,\omega),v_i(t;u_0,v_0,\omega))\}_{i\in\Z}$ with $(u(0;u_0,v_0,\omega),v(0;u_0,v_0,\omega))=(u_0,v_0)$. Note that, if $u_0\in l^{\infty,+}(\Z)$, $v_0\in l^{\infty,+}(\Z)$, then $(u(t;u_0,v_0,\omega),v(t;u_0,v_0,\omega))$ exists for all $t\geq0$ and $(u(t;u_0,v_0,\omega),v(t;u_0,v_0,\omega))\in l^{\infty,+}(\Z)\times l^{\infty,+}(\Z)$ for all $t\geq 0$. A solution $(u(t;\omega),v(t;\omega))=\{u_i(t;\omega),v_i(t;\omega)\}_{i\in\Z}$ of \eqref{main-eqn} is called {\it spatially homogeneous} if $u_i(t)=u_j(t)$ and $v_i(t)=v_j(t)$ for all $i,j\in\Z$.

Note that \eqref{main-eqn} contains the following two sub-systems,
\begin{equation}\label{single specie 1}
\dot{u}_i(t) =u_{i+1}( t ) -2u_i( t) +u_{i-1}( t) +u_i( t)(a_1(\theta _t\omega)- b_1(\theta _t\omega)u_i( t)),
\end{equation}
and
\begin{equation}\label{single specie 2}
\dot{v}_i(t) =v_{i+1}( t ) -2v_i( t) +v_{i-1}( t) +v_i( t)(a_2(\theta _t\omega)-c_2(\theta _t\omega)v_i( t)).
\end{equation}


First we give some notations and assumption related to (\ref{main-eqn}). Let
$$
\underline{a}( \omega) =\liminf\limits_{t-s\rightarrow\infty}\frac{1}{t-s}\int_s^t{a( \theta _{\tau}\omega)}d\tau :=\lim\limits_{r\rightarrow \infty}\inf_{t-s\ge r}\frac{1}{t-s}\int_s^t{a( \theta _{\tau}\omega)}d\tau
$$\\
and
$$
\overline{a}(\omega) =\limsup\limits_{t-s\rightarrow \infty}\frac{1}{t-s}\int_s^t{a( \theta _{\tau}\omega)}d\tau :=\lim\limits_{r\rightarrow \infty}\sup_{t-s\ge r}\frac{1}{t-s}\int_s^t{a( \theta _{\tau}\omega)}d\tau,
$$\\
where $a(\omega)$ could be $a_i(\omega)$, $b_i(\omega)$, $c_i(\omega)$ $(i=1$ or $2)$ or any similar function.
We call $\underline{a}(\cdot)$ and $\overline{a}(\cdot)$ the least mean and the greatest mean of $a(\cdot)$, respectively. It's easy to get that \\
$$ \underline{a}(\theta _t\omega ) =\underline{a}(\omega) \ \ \text{and } \ \ \overline{a}( \theta _t\omega) =\overline{a}( \omega) \ \ \mbox{ for all } t\in\mathbb{R},  $$
and
$$
 \underline{a}(\omega) =\liminf\limits_{t,s\in \mathbb{Q},t-s\rightarrow \infty}\frac{1}{t-s}\int_s^t{a( \theta _{\tau}\omega)}d\tau \ \ \ \text{and\\\ } \ \ \overline{a}(\omega) =\limsup\limits_{t,s\in \mathbb{Q},t-s\rightarrow \infty}\frac{1}{t-s}\int_s^t{a( \theta _{\tau}\omega)}d\tau .
$$\\
Then  $\underline{a}(\omega)$  and $\overline{a}(\omega)$  are measurable in $\omega$. The ergodicity of the metric dynamical system $(\Omega,\mathcal{F},\mathbb{P},\{\theta_t\}_{t\in\mathbb{R}})$ implies that, there are $\underline{a},\overline{a}\in\mathbb{R}$ and a measurable subset $\Omega_0\subset\Omega$ with  $\mathbb{P}(\Omega_0)=1$ such that \\
$$\left\{ \begin{array}{l} 	\theta _t\Omega _0=\Omega _0\ \ \forall t\in \mathbb{R}\\ 	 \liminf\limits_{t-s\rightarrow \infty}\frac{1}{t-s}\int_s^t{a( \theta _{\tau}\omega)}d\tau =\underline{a}\ \ \forall \omega \in \Omega _0\\ 	 \limsup\limits_{t-s\rightarrow \infty}\frac{1}{t-s}\int_s^t{a( \theta _{\tau}\omega)}d\tau =\overline{a}\ \ \forall \omega \in \Omega _0,\\ \end{array} \right. $$
That is, $\underline{a}(\omega)$ and $\overline{a}(\omega)$ are independent of $\omega$ in a subset of $\Omega$ of full measure (see Lemma \ref{L-lemma}).

Throughout this paper, we assume that the trivial solution $(0,0)$ of \eqref{main-eqn} is unstable with respect to perturbation in $ l^{\infty}(\Z)\times l^{\infty}(\Z)$, i.e.

\medskip
\noindent{\bf (H1)}  $ \underline{a_i}(\omega)= \underset{t-s\rightarrow \infty}{\lim\inf}\frac{1}{t-s}\int_s^t{a_i( \theta _{\tau}\omega)}d\tau>0$ $(i=1,2)$ for a.e. $\omega\in\Omega$. \medskip

Note that  $\bf (H1)$ implies that \eqref{main-eqn} has two semi-trivial spatially homogeneous positive solutions $(u^\ast(t;\omega),0):=(\phi^*(\theta_t\omega),0)\in {\rm Int}~l^{\infty,+}(\Z)\times l^{\infty,+}(\Z)$ and $(0,v^\ast(t;\omega)):=(0,\psi^*(\theta_t\omega)) \in l^{\infty,+}(\Z)\times{\rm Int}~ l^{\infty,+}(\Z)$ for some random equilibria $\phi^*$ and $\psi^*$, where $u^*(t;\omega)=\phi^*(\theta_t\omega)$ is the unique spatially homogeneous positive solution of \eqref{single specie 1}, and  $v^*(t;\omega)=\psi^*(\theta_t\omega)$ is the unique spatially homogeneous positive solution of \eqref{single specie 2} (see \cite[Theorem 1.1] {CaSh17} and \cite[Theorem A]{MiSh04}).

We also assume that

\medskip
\noindent{\bf (H2)}  $ (0,v^\ast(t;\omega))$ is unstable in $l^{\infty,+}(\Z)\times l^{\infty,+}(\Z)$, i.e. $\underline{a_1( \omega ) -c_1( \omega )v^*( \cdot;\omega )}>0$. $ (u^\ast(t;\omega),0)$ is linearly and globally stable in $l^{\infty,+}(\Z)\times l^{\infty,+}(\Z)$, i.e. $\overline{a_2( \omega ) -b_2( \omega )u^*( \cdot;\omega )}<0$, and for any $(u_0,v_0)\in \l^{\infty,+}(\Z)\times l^{\infty,+}(\Z)$ with $u_0\neq0$ and a.e. $\omega\in\Omega$, $u_i(t;u_0,v_0,\theta_{t_0}\omega)-u^*(t+t_0;\omega)\rightarrow0$ and $v_i(t;u_0,v_0,\theta_{t_0}\omega)\rightarrow0$ as $t\rightarrow\infty$ uniformly in $i\in\mathbb{Z}$ and $t_0\in\R$.

We remark that if $\underline{a_1(\omega ) -c_1( \omega )v^*( \cdot;\omega )}>0$, then  $ (0,v^\ast(t;\omega))$ is unstable in $l^{\infty,+}(\Z)\times l^{\infty,+}(\Z)$, and if $\overline{a_2( \omega ) -b_2( \omega )u^*( \cdot;\omega )}<0$, then $ (u^\ast(t;\omega),0)$ is locally stable in $l^{\infty,+}(\Z)\times l^{\infty,+}(\Z)$, and if $\underline{a_i}(\omega)>0$ $(i=1,2)$, $a^{\omega}_{1L}>\frac{c^{\omega}_{1M}a^{\omega}_{2M}}{c^{\omega}_{2L}}$ and $a^{\omega}_{2M}\leq\frac{a^{\omega}_{1L}b^{\omega}_{2L}}{b^{\omega}_{1M}}$ for any $\omega\in\Omega$, then $ (u^\ast(t;\omega),0)$ is globally stable and $ (0,v^\ast(t;\omega))$ is unstable in $l^{\infty,+}(\Z)\times l^{\infty,+}(\Z)$, where $a^{\omega}_{iL}=\underset{t\in\mathbb{R}}{\text{inf}}a_i( \theta _t\omega )$, $a^{\omega}_{iM}=\underset{t\in\mathbb{R}}{\text{sup}}a_i( \theta _t\omega )$  and $b^{\omega}_{iL}$, $b^{\omega}_{iM}$, $c^{\omega}_{iL}$, $c^{\omega}_{iM}$ are defined similarly (This can be proved similarly as \cite[Proposition 2.4]{Ba18}).

Now we present the third standing hypothesis.

\medskip
\noindent{\bf (H3)} For any $\omega\in\Omega$, $\underset{t\in\mathbb{R}}{\inf}b_2(\theta_t\omega) >0$, $b_i(\theta_t\omega) \geq c_i(\theta_t\omega)$ $(i=1,2)$ and
$$a_1( \theta _t\omega ) -c_1( \theta _t\omega ) v^*( t;\omega ) \geq a_2( \theta _t\omega ) -2c_2( \theta _t\omega ) v^*( t;\omega )+b_2(\theta _t\omega)v^*( t;\omega )\ \ \mathrm{for\ any}\ t\in\mathbb{R}.$$

Under the assumptions {\bf (H1)}-{\bf (H3)}, one of the most interesting dynamical problems is to study the existence of random transition front (generalized traveling wave) solutions connecting $(u^*(t;\omega),0)$ and $(0,v^*(t;\omega))$ for \eqref{main-eqn}. To do so, we first transform \eqref{main-eqn} to a cooperative system via the following standard change of variables,
$$\tilde{u}_i=u_i,\ \ \tilde{v}_i=v^*(t;\omega)-v_i.$$
Dropping the tilde, \eqref{main-eqn} is transformed into
\begin{equation}\label{main-eqn-trans1}
\left\{
\begin{aligned}
\dot{u}_i &=Hu_i+u_i(a_1(\theta _t\omega)- b_1(\theta _t\omega)u_i-c_1(\theta _t\omega)(v^*(t;\omega)-v_i)) ,\\
\dot{v}_i &=Hv_i+b_2(\theta _t\omega)(v^*(t;\omega)-v_i)u_i+v_i(a_2(\theta _t\omega)- 2c_2(\theta _t\omega)v^*(t;\omega)+c_2(\theta _t\omega)v_i),
\end{aligned}
\right.
\end{equation}
where  $$Hu_i(t):=u_{i+1}(t)-2u_i(t)+u_{i-1}(t), \quad i\in \Z,\, t\in\R.$$


It is clear that \eqref{main-eqn-trans1} is cooperative in the region
 $u_i(t)\ge0$ and $0\le v_i(t)\le v^*(t;\omega)$, and the trivial solution $(0,0)$ of \eqref{main-eqn} becomes $(0,v^*(t;\omega))$, the semitrivial solutions $(0,v^*(t;\omega))$ and $(u^*(t;\omega),0)$ of \eqref{main-eqn} becomes $(0,0)$ and $(u^*(t;\omega),v^*(t;\omega))$, respectively. To study the random transition front solutions of \eqref{main-eqn} connecting $(u^*(t;\omega),0)$ and $(0,v^*(t;\omega))$ is then equivalent to study the random transition front solutions of \eqref{main-eqn-trans1} connecting $(u^*(t;\omega),v^*(t;\omega))$ and $(0,0)$.

We denote $(u(t;u^0,v^0,\omega),v(t;u^0,v^0,\omega))=\{(u_i(t;u^0,v^0,\omega),v_i(t;u^0,v^0,\omega))\}_{i\in\Z}$ as the solution of \eqref{main-eqn-trans1} with $(u(0;u^0,v^0,\omega),v(0;u^0,v^0,\omega))=(u^0,v^0)\in l^\infty(\Z)\times l^\infty(\Z)$. For any $(u^1,u^2)\in l^\infty(\Z)\times l^\infty(\Z)$ and $(v^1,v^2)\in l^\infty(\Z)\times l^\infty(\Z)$, the relation $(u^1,u^2) <  (v^1,v^2)$ $((u^1,u^2) \leq (v^1,v^2)$ resp.) is to be understood componentwise: $u^i < v^i$ $(u^i \leq v^i) $ for each $i$. Other relations like $``\mathrm{max}"$, $``\mathrm{min}"$, $``\mathrm{sup}"$, $``\mathrm{inf}"$ can be similarly understood. Then it is clear that, if $(u^0,v^0)\geq (0,0)$, then $(u(t;u^0,v^0,\omega),v(t;u^0,v^0,\omega))$ exists for all $t\geq0$ and $(u(t;u^0,v^0,\omega),v(t;u^0,v^0,\omega))\geq(0,0)$ for all $t\geq 0$ (see Proposition \ref{comparison}). A solution $(u(t;\omega),v(t;\omega))=\{(u_i(t;\omega),v_i(t;\omega))\}_{i\in\mathbb{Z}}$ of (\ref{main-eqn-trans1}) is called an \textit{entire solution} if it is a solution of (\ref{main-eqn-trans1}) for $t\in\mathbb{R}$.


\begin{definition} [Random transition front] An entire solution $(u(t;\omega),v(t;\omega))$ is called a   \textit{random transition front} or a \textit{random generalized traveling wave} of \eqref{main-eqn-trans1} connecting $(0,0)$ and $(u^*(t;\omega),v^*(t;\omega))$
if for  a.e. $\omega\in\Omega$,
$$
 (u_i( t;\omega ),v_i( t;\omega )) =({\varPhi} ( i-\int_0^t{c( s;\omega) ds},\theta _t\omega),{\varPsi}( i-\int_0^t{c( s;\omega) ds},\theta _t\omega))
 $$
for some $ {\varPhi}( x,\omega)$, ${\varPsi}( x,\omega)$ $(x\in\R)$ and $c( t;\omega) $, where
$ {\varPhi}( x,\omega)$, ${\varPsi}( x,\omega)$ and $c( t;\omega) $ are measurable in $\omega$,
and for a.e. $\omega\in\Omega$,

$$(0,0)<({\varPhi}( x,\omega),{\varPsi}( x,\omega) <(u^*(t;\omega),v^*(t;\omega)),$$

\begin{equation*}
\begin{aligned}
&\underset{x\rightarrow -\infty}{\lim}({\varPhi}( x,\theta _t\omega),{\varPsi}( x,\theta _t\omega)) =(u^*(t;\omega),v^*(t;\omega)),\mbox{ and }\\
&\underset{x\rightarrow \infty}{\lim}({\varPhi}( x,\theta _t\omega),{\varPsi}( x,\theta _t\omega)) =(0,0)\ \ \ \ \text{ uniformly\ in\ }t\in \mathbb{R}.
\end{aligned}
\end{equation*}
\end{definition}

Suppose that $(u(t;\omega),v(t;\omega))=\{(u_i(t;\omega),v_i(t;\omega))\}_{i\in\mathbb{Z}}$ with $(u_i( t;\omega),v_i( t;\omega)) =({\varPhi}( i-\int_0^t{c( s;\omega) ds},\theta _t\omega ),{\varPsi}( i-\int_0^t{c( s;\omega) ds},\theta _t\omega ))$ is a \textit{random transition front} of (\ref{main-eqn-trans1}). If  $ {\varPhi} (x,\omega)$ and ${\varPsi} (x,\omega)$ are non-increasing with respect to $x$ for a.e. $\omega\in\Omega$ and all $x\in\mathbb{R}$, then $(u(t;\omega),v(t;\omega))$ is said to be a  \textit{monotone random transition front}. If there is $\overline{c}_{\inf}\in\mathbb{R}$ such that for a.e. $\omega\in\Omega$,
$$ \underset{t-s\rightarrow \infty}{\lim\inf}\frac{1}{t-s}\int_s^t{c( \tau ;\omega ) d\tau}=\overline{c}_{\inf}, $$
then $\overline{c}_{\inf}$ is called its \textit{least mean speed}.

Note that the ergodicity of the metric dynamical system $(\Omega,\mathcal{F},\mathbb{P},\{\theta_t\}_{t\in\mathbb{R}})$ implies that  $\underline{a_1(\omega ) -c_1( \omega )v^*( \cdot;\omega )}=\underset{t-s\rightarrow \infty}{\lim\text{inf}}\frac{1}{t-s}\int_s^t{( a_1( \theta _{\tau}\omega ) -c_1( \theta _{\tau}\omega ) v^*( \tau ;\omega ) )}d\tau$ is independent of $\omega$ in a subset $\Omega_{0}\subset \Omega$ of full measure. We denote $\underline{\lambda} =\underline{a_1(\omega ) -c_1( \omega )v^*( \cdot;\omega )}$ for $\omega\in\Omega_0$.
For given $\mu>0$, let
$$
 c_0:=\inf \limits_{\mu>0}\frac{e^{\mu}+e^{-\mu}-2+\underline{\lambda}}{\mu},
 $$
 By \cite[Lemma 5.1]{CaSh17},
 there is a unique $\mu^*>0$ such that
 $$ c_0=\frac{e^{\mu ^*}+e^{-\mu ^*}-2+\underline{\lambda}}{\mu ^*} $$
  and for any $\gamma>c_0$, the equation $ \gamma=\frac{e^{\mu }+e^{-\mu}-2+\underline{\lambda}}{\mu} $ has exactly two positive solutions for $\mu$.

Now we are in a position to state the main results on the existence and non-existence of random transition fronts of two species cooperative lattice systems in random media.

\begin{theorem}\label{exist-thm}
Assume {\bf(H1)}-{\bf(H3)} hold. Then we have
\begin{itemize}
\item[(i)]
For any given $\gamma>c_0$, there is a monotone random transition front of  \eqref{main-eqn-trans1} with least mean speed $\overline{c}_{\inf}=\gamma$. More precisely,  for any given $\gamma>c_0$,  let $0<\mu<{\mu}^*$ be such that $\frac{e^{\mu}+e^{-\mu}-2+\underline{\lambda}}{\mu}=\gamma$. Then \eqref{main-eqn-trans1} has a monotone random transition front $(u(t;\omega),v(t;\omega))=\{(u_i(t;\omega),v_i(t;\omega))\}_{i\in\mathbb{Z}}$ with $u_i( t;\omega) ={\varPhi}( i-\int_0^t{c( s;\omega) ds},\theta _t\omega )$ and $v_i( t;\omega)={\varPsi}( i-\int_0^t{c( s;\omega) ds},\theta _t\omega )$, where $c(t;\omega) =\frac{e^{\mu}+e^{-\mu}-2+a_1( \theta _t\omega ) -c_1( \theta _t\omega ) v^*( t;\omega )}{\mu}$ and hence $\overline{c}_{\inf}=\frac{e^{\mu}+e^{-\mu}-2+\underline{\lambda}}{\mu}=\gamma$.
Moreover, for any $\omega\in \Omega_0$,
$$\underset{x\rightarrow -\infty}{\lim}({\varPhi}( x,\theta _t\omega),{\varPsi}( x,\theta _t\omega)) =(u^*(t;\omega),v^*(t;\omega))$$
and
$$\underset{x\rightarrow \infty}{\lim}({\varPhi}( x,\theta _t\omega),{\varPsi}( x,\theta _t\omega)) =(0,0)$$
uniformly in $t\in \mathbb{R}$.
\item[(ii)]
There is no random transition front of \eqref{main-eqn-trans1} with least mean speed less than $c_0$.
\end{itemize}
\end{theorem}

\begin{remark}
\begin{itemize}
\item[(i)] 
When $a_i, b_i, c_i$ $(i=1,2)$ are constants, our existence result of the transition front is consistent with the result obtained in \cite[Theorem 1, Theorem 4]{GuWu12}. Also, we get the non-existence result of the transition front.
\item[(ii)] We leave as an open problem the case $\overline{c}_{\inf}=c_0$, that is, the existence of random transition front of \eqref{main-eqn-trans1} with least mean speed $\overline{c}_{\inf}=c_0$.
\end{itemize}
\end{remark}

The rest of the paper is organized as follows. In Section 2, we establish the comparison
principle for sub-solutions and super-solutions of  \eqref{main-eqn-trans1} and prove some basic properties and fundamental lemmas to be used in later section. We prove the existence and non-existence of random transition fronts after constructing appropriate sub-solutions and super-solutions of \eqref{main-eqn-trans1} in Section 3.

\section{Preliminary}

In this section, we present some preliminary materials to be used in later sections. We
first present a comparison principle for sub-solutions and super-solutions of \eqref{main-eqn-trans1} and prove the convergence of solutions on compact subsets. Next, we present some useful lemmas including a technical lemma.

Consider first the following space continuous version of \eqref{main-eqn-trans1},
\begin{equation}\label{main-trans1-continuous}
\left\{
\begin{aligned}
\pa_t u=&Hu+u(a_1(\theta _t\omega)- b_1(\theta_t\omega)u-c_1(\theta _t\omega)(v^*(t;\omega)-v)) ,\\
\pa_t v=&Hv+b_2(\theta _t\omega)(v^*(t;\omega)-v)u+v(a_2(\theta _t\omega)- 2c_2(\theta _t\omega)v^*(t;\omega)+c_2(\theta _t\omega)v),
\end{aligned}
\right.
\end{equation}
where
\begin{equation*}
u=u(x,t),~v=v(x,t),\end{equation*}
and
\begin{equation*}
Hu(x,t):=u(x+1,t)+u(x-1,t)-2u(x,t),\ \ \quad x\in \R,\, t\in\R.
\end{equation*}

Let
$$
l^\infty(\R)=\{u:\R\to\R\,:\, \sup_{x\in\R}|u(x)|<\infty\}
$$
with norm $\|u\|=\sup_{x\in\R}|u(x)|$, and
$$l^{\infty,+}(\R)=\{u\in l^\infty(\R)\,:\, \inf_{x\in\R}u(x)\ge 0\}.$$
For $u,v\in l^\infty(\R)$, we define
$$
u\ge v \quad {\rm if}\quad u-v\in l^{\infty,+}(\R).
$$

For any $(u_0,v_0)\in l^\infty(\R)\times l^\infty(\R)$, let $(u(x,t;u_0,v_0,\omega),v(x,t;u_0,v_0,\omega))$ be the solution of \eqref{main-trans1-continuous} with $(u(x,0;u_0,v_0,\omega),v(x,0;u_0,v_0,\omega))=(u_0(x),v_0(x))$. Recall that
for any $(u^0,v^0)\in l^\infty(\Z)\times l^\infty(\Z)$, $(u(t;u^0,v^0,\omega),v(t;u^0,v^0,\omega))=\{(u_i(t;u^0,v^0,\omega),v_i(t;u^0,v^0,\omega))\}_{i\in\Z}$ is the solution of \eqref{main-eqn-trans1} with $(u_i(0;u^0,v^0,\omega),v_i(0;u^0,v^0,\omega))=(u^0_i,v^0_i)$ for $i\in\Z$. For any $(u^1,u^2)$, $(v^1,v^2)\in l^\infty(\R)\times l^\infty(\R)$, the relation $(u^1,u^2) <  (v^1,v^2)$ $((u^1,u^2) \leq (v^1,v^2)$ resp.) is also to be understood componentwise: $u^i < v^i$ $(u^i \leq v^i) $ for each $i$.

Let
$$f(t,u,v,\omega)=u(a_1(\theta _t\omega)- b_1(\theta_t\omega)u-c_1(\theta _t\omega)(v^*(t;\omega)-v))$$
and
$$g(t,u,v,\omega)=b_2(\theta _t\omega)(v^*(t;\omega)-v)u+v(a_2(\theta _t\omega)- 2c_2(\theta _t\omega)v^*(t;\omega)+c_2(\theta _t\omega)v).$$
A pair of function $(u(x,t;\omega),v(x,t;\omega))$ on $\R\times[0,T)$ which is continuous in $t$
is called a {\it super-solution} or {\it sub-solution} of \eqref{main-trans1-continuous} (resp. \eqref{main-eqn-trans1}) if for a.e. $\omega\in\Omega$ and any given $x\in\R$ (resp. $x\in\Z$), $u(x,t;\omega)$ and $v(x,t;\omega)$ are absolutely continuous in $t\in [0,T)$, and
$$
\left\{ \begin{array}{l}
	u_t( x,t;\omega ) \ge  Hu( x,t;\omega ) +f( t,u,v,\omega )\\
	v_t( x,t;\omega ) \ge  Hv( x,t;\omega ) +g( t,u,v,\omega )\\
\end{array} \right. \ \ \ \ \mathrm{for\ \ a.e. }\  t\in [ 0,T)
$$
or
$$
\left\{ \begin{array}{l}
	u_t( x,t;\omega )\le Hu( x,t;\omega ) +f( t,u,v,\omega )\\
	v_t( x,t;\omega )\le Hv( x,t;\omega ) +g( t,u,v,\omega )\\
\end{array} \right. \ \ \ \ \mathrm{for\ \ a.e. }\ t\in [ 0,T) .
$$
A pair of function is said to be a generalized super-solution (resp. sub-solution)
if it is the infimum (resp. supremum) of a finite number of super-solutions (resp. sub-solutions).

Now we are in a position to present a comparison principle for solutions of \eqref{main-trans1-continuous}, the comparison principle for solutions of \eqref{main-eqn-trans1} can be proved similarly.

\begin{proposition}[Comparison principle]
\label{comparison}
\begin{itemize}
\item[(1)]
Suppose that $(u_1(x,t;\omega),v_1(x,t;\omega))$ is a bounded sub-solution of \eqref{main-trans1-continuous} on $[0,T)$ and that $(u_2(x,t;\omega),v_2(x,t;\omega))$ is a bounded super-solution of \eqref{main-trans1-continuous}  on $[0,T)$ and $(u_i(x,t;\omega),v_i(x,t;\omega))$ $\in$ $[ 0,u^*( t;\omega ) ] \times [ 0,v^*( t;\omega ) ]$ $( i=1,2 )$ for $x\in\R$ and $t\in [0,T)$. If $(u_1(\cdot,0;\omega),v_1( \cdot,0;\omega))\le(u_2( \cdot,0;\omega ) ,v_2( \cdot,0;\omega))$, then
$$(u_1(\cdot,t;\omega),v_1( \cdot,t;\omega))\le(u_2( \cdot,t;\omega ) ,v_2( \cdot,t;\omega)) \quad \mbox{ for }\  t\in[0,T).$$

\item[(2)]  Suppose that $(u_i(x,t;\omega),v_i(x,t;\omega))$ $\in$ $[ 0,u^*( t;\omega ) ] \times [ 0,v^*( t;\omega ) ]$ $(i=1,2)$ are bounded and satisfy that for any given $x\in\R$,
 $u_i(x,t;\omega)$, $v_i(x,t;\omega)$ $(i=1,2)$ are absolutely continuous in $t\in[0,\infty)$, and
 \begin{align*}
 &\pa_t u_2(x,t;\omega)-(Hu_2(x,t;\omega)+f(t,u_2,v_2,\omega))\\
 &>\pa_t u_1(x,t;\omega)-(Hu_1(x,t;\omega)+f(t,u_1,v_1,\omega)),
 \end{align*}
 \begin{align*}
 &\pa_t v_2(x,t;\omega)-(Hv_2(x,t;\omega)+g(t,u_2,v_2,\omega))\\
 &>\pa_t v_1(x,t;\omega)-(Hv_1(x,t;\omega)+g(t,u_1,v_1,\omega))
 \end{align*}
for $t>0$. Moreover, suppose that $(u_2(\cdot,0;\omega),v_2(\cdot,0;\omega))\geq (u_1(\cdot,0;\omega),v_1(\cdot,0;\omega))$. Then $(u_2(\cdot,t;\omega),v_2(\cdot,t;\omega))> (u_1(\cdot,t;\omega),v_1(\cdot,t;\omega))$ for  $t>0$.
\end{itemize}
\end{proposition}

\begin{proof}
(1)  Let $Q_1( x,t;\omega ) =e^{ct}( u_2( x,t;\omega ) -u_1( x,t;\omega))$ and $Q_2( x,t;\omega ) =e^{ct}( v_2( x,t;\omega ) -v_1( x,t;\omega))$, where $c:=c(\omega)$ is to be determined later. Then there is a measurable subset $\bar{\Omega}$ of $\Omega$ with $\mathbb{P}(\bar{\Omega})=0$ such that for any $\omega\in\Omega\setminus\bar\Omega$, $Q_1( x,t;\omega )$ and $Q_2( x,t;\omega )$ satisfy
\begin{align}
\label{difference}
\left\{ \begin{array}{l}
	\partial_t Q_1\ge Q_1(x+1,t;\omega)+Q_1(x-1,t;\omega) +a_1( x,t;\omega) Q_1( x,t;\omega)+b_1( x,t;\omega) Q_2( x,t;\omega ),\\
	\partial_t Q_2\ge Q_2( x+1,t;\omega ) +Q_2( x-1,t;\omega ) +a_2( x,t;\omega) Q_1( x,t;\omega ) +b_2( x,t;\omega ) Q_2( x,t;\omega ),\\
\end{array} \right.
\end{align}
where
\begin{align*}
a_1(x,t;\omega ) =c-2+f_u( t,u_{1}^{*},v_{1}^{*},\omega),\ \ b_1( x,t;\omega ) =f_v( t,u_{1}^{*},v_{1}^{*},\omega),\\
a_2( x,t;\omega ) =g_u( t,u_{2}^{*},v_{2}^{*},\omega),\ \ b_2(x,t;\omega ) =c-2+g_v( t,u_{2}^{*},v_{2}^{*},\omega)
\end{align*}
for some $u_i^*=u^*_i(x,t;\omega)$ $(i=1,2)$ between $u_1(x,t;\omega)$ and $u_2(x,t;\omega)$ and some $v_i^*=v^*_i(x,t;\omega)$ $(i=1,2)$ between $v_1(x,t;\omega)$ and $v_2(x,t;\omega)$.

Since system \eqref{main-trans1-continuous} is cooperative in $[ 0,u^*( t;\omega ) ] \times [ 0,v^*( t;\omega ) ]$, we have $b_1(x,t;\omega)\geq0$ and $a_2(x,t;\omega)\geq0$.
By the boundedness of $u_i(x,t;\omega)$ and $v_i(x,t;\omega)$ $(i=1,2)$, we can choose $c=c(\omega)>0$ such that $b_2(x,t;\omega)\geq0$ and $a_1(x,t;\omega)\geq0$.

We claim that $Q_i(x,t;\omega)\geq0$ $(i=1,2)$ for $x\in\R$ and $t\in[0,T]$.
Let $p_0(\omega) :=\underset{i=1,2}{\max}~\underset{( x,t) \in \mathbb{R}\times [ 0,T ]}{\max}\{ a_i( x,t;\omega ) ,b_i( x,t;\omega ) \}$. It suffices to prove the claim for $x\in\R$ and $t\in(0,T_0]$ with $T_0=\min \{ T,\frac{1}{2( 1+p_0(\omega) )}\}.$ Assume that there are some $\tilde{x}\in\R$ and $\tilde{t}\in(0,T_0]$ such that $Q_1(\tilde{x},\tilde{t};\omega)<0$ or $Q_2(\tilde{x},\tilde{t};\omega)<0$.
Then there is $t^0\in(0,T_0)$ such that $$Q_1^{\inf}(\omega):=\underset{(x,t)\in\R\times[0,t^0]}{\inf}{Q_1(x,t;\omega)}<0\ \ \mathrm{or}\ \ Q_2^{\inf}(\omega):=\underset{(x,t)\in\R\times[0,t^0]}{\inf}{Q_2(x,t;\omega)}<0. $$
Without loss of generality, we assume that $Q^{\inf}_1(\omega)\leq Q^{\inf}_2(\omega)$.
Observe that there are $x_n\in\R$ and $t_n\in(0,t^0]$ such that
$$Q_1(x_n,t_n;\omega)\to Q^{\inf}_1(\omega)\quad\mbox{ as }\,n\to\infty.$$
By \eqref{difference}  and the fundamental theorem of calculus for Lebesgue integrals, we get
\begin{align}
&Q_1(x_n,t_n;\omega)-Q_1(x_n,0;\omega)\nonumber\\
&\geq\int^{t_n}_0[Q_1(x_n+1,t;\omega)+Q_1(x_n-1,t;\omega)+a_1(x_n,t;\omega)Q_1(x_n,t;\omega)\nonumber\\
& \ \ \ +b_1(x_n,t;\omega)Q_2(x_n,t;\omega)]dt \nonumber\\
&\geq\int^{t_n}_0[2Q^{\inf}_1(\omega)+a_1(x_n,t;\omega)Q^{\inf}_1(\omega)+b_1(x_n,t;\omega)Q^{\inf}_2(\omega)]dt \nonumber\\
&\geq\int^{t_n}_0[2Q^{\inf}_1(\omega)+2p_0(\omega)Q^{\inf}_1(\omega)]dt \nonumber\\
&\geq 2(1+p_0(\omega))t^0Q^{\inf}_1(\omega)\quad\quad\mbox{ for }\,n\geq1.\nonumber
\end{align}
Note that $Q_1(x_n,0;\omega)\geq0$, we then have
$$Q_1(x_n,t_n;\omega)\geq 2(1+p_0(\omega))t^0Q^{\inf}_1(\omega)\quad\quad\mbox{ for }\,n\geq1.$$
Letting $n\to\infty$, we obtain
$$Q^{\inf}_1(\omega)\geq 2(1+p_0(\omega))t^0Q^{\inf}_1(\omega)>Q^{\inf}_1(\omega).$$
A contradiction. Hence $Q_i(x,t;\omega)\geq0$ $(i=1,2)$ for $x\in \R$ and $t\in [0,T]$, which implies that $(u_1(x,t;\omega),v_1( x,t;\omega))\le(u_2( x,t;\omega ) ,v_2( x,t;\omega))$ for $\omega\in\Omega\setminus\bar\Omega$, $x\in\R$ and $t\in[0,T]$.

(2)  Since system \eqref{main-trans1-continuous} is cooperative in $[ 0,u^*( t;\omega ) ] \times [ 0,v^*( t;\omega ) ]$, then for $\omega\in\Omega\setminus\bar{\Omega}$, by the similar arguments as getting \eqref{difference}, we can find $c(\omega), \mu(\omega)>0$ such that for any given $x\in\R$,
$$\partial_tw(x,t;\omega)>w(x+1,t;\omega)+w(x-1,t;\omega)+\mu(\omega) w(x,t;\omega) \ \ \mathrm{for}\  \ t>0, $$
where $ w(x,t;\omega)=e^{c(\omega)t}(u_2(x,t,\omega)-u_1(x,t,\omega))$.
Thus we have that for any given $x\in\R$,
$$w(x,t;\omega)>w(x,0;\omega)+\int^t_0[w(x+1,s;\omega)+w(x-1,s;\omega)+\mu(\omega) w(x,s;\omega)]ds.$$
By the arguments in (1), $w(x,t;\omega)\geq0$ for all $x\in\R$ and $t\geq0$. It then follows that $w(x,t;\omega)>w(x,0;\omega)\geq0$ and and hence $u_2(x,t;\omega)>u_1(x,t;\omega)$ for $\omega\in\Omega\setminus\bar\Omega$, $x\in\R$ and $t>0$.
 Similarly, we can get that $v_2(x,t;\omega)>v_1(x,t;\omega)$ for $\omega\in\Omega\setminus\bar\Omega$, $x\in\R$ and $t>0$.
\end{proof}

\begin{proposition}
\label{convergence-lemma}
Suppose that $(u_n,v_n)\in l^{\infty,+}(\R)\times l^{\infty,+}(\R)$ $(n=1,2,\cdots)$ and $(u_0,v_0)\in l^{\infty,+}(\R)\times l^{\infty,+}(\R)$ with $\{\|u_n\|\}$, $\{\|v_n\|\}$ being bounded.
If $(u_n(x),v_n(x))\to (u_0(x),v_0(x))$ as $n\to\infty$
 uniformly in $x$ on bounded sets,
{ then for each $t>0$}, $(u(x,t;u_n,v_n,\theta_{t_0}\omega),v(x,t;u_n,v_n,$\\ $\theta_{t_0}\omega))\to (u(x,t;u_0,v_0,\theta_{t_0}\omega),v(x,t;u_0,v_0,\theta_{t_0}\omega))$ as $n\to\infty$ uniformly in $x$ on bounded sets and $t_0\in\R$.
\end{proposition}

\begin{proof}
Fix any $\omega\in\Omega$. Let
\begin{align*}
\left\{\begin{array}{l}
u^n(x,t;\theta_{t_0}\omega)=u(x,t;u_n,v_n,\theta_{t_0}\omega)-u(x,t;u_0,v_0,\theta_{t_0}\omega),\\
v^n(x,t;\theta_{t_0}\omega)=v(x,t;u_n,v_n,\theta_{t_0}\omega)-v(x,t;u_0,v_0,\theta_{t_0}\omega).\\
\end{array} \right.
\end{align*}
Then
\begin{align*}
\left\{ \begin{array}{l}
\partial_t u^{n}=Hu^n+a_{1}^{n}( x,t;\theta _{t_0}\omega)u^n+b_{1}^{n}( x,t;\theta _{t_0}\omega) v^n,\\
\partial_t v^{n}=Hv^n+a_{2}^{n}( x,t;\theta _{t_0}\omega)u^n+b_{2}^{n}( x,t;\theta _{t_0}\omega) v^n,\\
\end{array} \right.
\end{align*}
where
$$a_{1}^{n}( x,t;\theta _{t_0}\omega)=f_u(t,u_1^n(x,t;\theta_{t_0}\omega),v^n_1(x,t;\theta_{t_0}\omega),\theta_{t_0}\omega),$$
$$b_{1}^{n}( x,t;\theta _{t_0}\omega)=f_v(t,u_1^n(x,t;\theta_{t_0}\omega),v^n_1(x,t;\theta_{t_0}\omega),\theta_{t_0}\omega),$$
$$a_{2}^{n}( x,t;\theta _{t_0}\omega)=g_u(t,u_2^n(x,t;\theta_{t_0}\omega),v^n_2(x,t;\theta_{t_0}\omega),\theta_{t_0}\omega),$$
and
$$b_{2}^{n}( x,t;\theta _{t_0}\omega)=g_v(t,u_2^n(x,t;\theta_{t_0}\omega),v^n_2(x,t;\theta_{t_0}\omega),\theta_{t_0}\omega),$$
for some $u_1^n(x,t;\theta_{t_0}\omega)$, $u_2^n(x,t;\theta_{t_0}\omega)$ between $u(x,t;u_n,v_n,\theta_{t_0}\omega)$ and $u(x,t;u_0,v_0,\theta_{t_0}\omega)$, and some $v^n_1(x,t;\theta_{t_0}\omega))$, $v^n_2(x,t;\theta_{t_0}\omega))$ between $v(x,t;u_n,v_n,\theta_{t_0}\omega)$ and $v(x,t;u_0,v_0,\theta_{t_0}\omega)$.

Take a $\rho>0$. Let
$$
X(\rho)=\{(u,v): \R \to \R^2 \ |\,\  (u(\cdot) e^{-\rho |\cdot|},v(\cdot) e^{-\rho |\cdot|})\in l^\infty(\R)\times l^\infty(\R)\}
$$
with norm $\|(u,v)\|_{X(\rho)}= \underset{x\in\R}{\text{sup}}( |u( x )| +|v( x )| ) e^{-\rho | x |}
$.
Observe that
$$(H,H):X(\rho)\to X(\rho),$$
$$(H,H)(u,v)=(Hu,Hv),$$
is a bounded linear operator. Note also that $a_{i}^{n}( x,t;\theta _{t_0}\omega)$ and $b_{i}^{n}( x,t;\theta _{t_0}\omega)$ are uniformly bounded $(i=1,2)$. Then there are $M>0$ and $\alpha>0$ such that
$$
\|e^{(H,H)t}\|_{X(\rho)}\leq M e^{\alpha t}
$$
and $|a_{i}^{n}( x,t;\theta _{t_0}\omega)|\leq M$, $ |b_{i}^{n}( x,t;\theta _{t_0}\omega)|\leq M$.
Hence,
\begin{align*}
&(u^n(\cdot,t;\theta_{t_0}\omega),v^n(\cdot,t;\theta_{t_0}\omega))\\&=e^{(H,H)t}(u^n(\cdot,0;\theta_{t_0}\omega),v^n(\cdot,0;\theta_{t_0}\omega))\\
&\ \ \ +\int_0^te^{(H,H)(t-\tau)}[a^n_1(\cdot,\tau;\theta_{t_0}\omega)u^n(\cdot,\tau;\theta_{t_0}\omega)+b^n_1(\cdot,\tau;\theta_{t_0}\omega)v^n(\cdot,\tau;\theta_{t_0}\omega),\\
&\ \ \ \ \ \ \ \ \ \ \  a^n_2(\cdot,\tau;\theta_{t_0}\omega)u^n(\cdot,\tau;\theta_{t_0}\omega)+b^n_2(\cdot,\tau;\theta_{t_0}\omega)v^n(\cdot,\tau;\theta_{t_0}\omega)]d\tau
\end{align*}
and then
\begin{align*}
&\|(u^n(\cdot,t;\theta_{t_0}\omega),v^n(\cdot,t;\theta_{t_0}\omega))\|_{X(\rho)}\\
&\leq M e^{\alpha t}\|(u^n(\cdot,0;\theta_{t_0}\omega),v^n(\cdot,0;\theta_{t_0}\omega))\|_{X(\rho)}\\
&\ \ \ +M^2\int_0^ t e^{\alpha(t-\tau)}\|(u^n(\cdot,\tau;\theta_{t_0}\omega),v^n(\cdot,\tau;\theta_{t_0}\omega))\|_{X(\rho)}d\tau.
\end{align*}
By Gronwall's inequality,
$$ \|(u^n(\cdot,t;\theta_{t_0}\omega),v^n(\cdot,t;\theta_{t_0}\omega))\|_{X(\rho)}\leq e^{( \alpha +M^2)t}M\|(u^n(\cdot,0;\theta_{t_0}\omega),v^n(\cdot,0;\theta_{t_0}\omega))\|_{X(\rho)}. $$
Note that $\|(u^n(\cdot,0;\theta_{t_0}\omega),v^n(\cdot,0;\theta_{t_0}\omega))\|_{X(\rho)}\to 0$ as $n\to\infty$ uniformly in $t_0\in\R$. It then follows that
$$
(u^n(x,t;\theta_{t_0}\omega),v^n(x,t;\theta_{t_0}\omega))\to(0,0) \quad {\rm as}\quad
n\to\infty
$$
uniformly in $x$ on bounded sets and $t_0\in\R$.
The Proposition thus follows.
\end{proof}

Now we present some lemmas including the technical results.

\begin{lemma}
\label{L-lemma}
$\underline{a}(\cdot), a(\cdot), \overline{a}(\cdot)\in\L^1(\Omega, \mathcal{F}, \mathbb{P})$. Also $\underline{a}(\omega)$ and $\overline{a}(\omega)$ are independent of $\omega$ for a.e. $\omega\in\Omega$.
\end{lemma}

\begin{proof}
 It follows from \cite[Lemma 2.1]{Sh18}.
\end{proof}

\begin{lemma}\label{technical}
Suppose that for $\omega\in\Omega$, $a^\omega(t)=a(\theta_t\omega)\in C(\mathbb{R},(0,\infty))$. Then for a.e. $\omega\in\Omega$,
$$\underline{a} =\sup\limits_{A\in W_{loc}^{1,\infty}( \mathbb{R} ) \cap L^{\infty}( \mathbb{R} )}\essinf\limits_{t \in \mathbb{R}}( A^{\prime}+a^{\omega} )(t). $$
\end{lemma}

\begin{proof}
It follows from \cite[Lemma 2.2] {Sh18} and Lemma \ref{L-lemma}.
\end{proof}

Note that by (H3) there is a strictly positive solution $h(t;\omega)$ of
$$
\frac{dv}{dt}-(a_2(\theta _t\omega) -2c_2( \theta _t\omega ) v^*( t;\omega )) v-b_2( \theta _t\omega) v^*( t;\omega) =-( a_1( \theta _t\omega) -c_1( \theta _t\omega)v^*(t;\omega )) v.
$$

\begin{lemma}
\label{sub-solution-lemma}
Let $\omega\in\Omega_0$ and $0<\sigma\ll 1$. Then for any $\mu$, $\tilde{\mu}$ with $ 0<\mu <\tilde{\mu}<\min\{ 2\mu ,\mu ^* \},$
there exist $ \{  t_k \} _{k\in \mathbb{Z}} $ with $ t_k<t_{k+1} $
and $ \underset{k\rightarrow \pm \infty}{\lim}t_k=\pm \infty  $, $ A_{\omega}\in W_{loc}^{1,\infty}( \mathbb{R}) \cap L^{\infty}( \mathbb{R})$ with $ A_{\omega}(\cdot) \in C^1(( t_k,t_{k+1}))  $ for $ k\in \mathbb{Z} $, and $ d_{\omega}>0 $ such that for any $ d\ge d_{\omega} $  the functions
\begin{equation*}
\setlength\abovedisplayskip{1pt}
\setlength\belowdisplayskip{1pt}
\begin{aligned}
&\tilde{u}( x,t, \omega) :=e^{-\mu ( x-\int_0^t{c( s ;\omega ,\mu ) ds} )}-de^{( \frac{\tilde{\mu}}{\mu}-1 ) A_{\omega}( t ) -\tilde{\mu}( x-\int_0^t{c( s ;\omega ,\mu ) ds} )},\\
&\tilde{v}( x,t, \omega) :=\sigma e^{-\mu ( x-\int_0^t{c( s ;\omega ,\mu ) ds} )}h(t;\omega)-\sigma de^{( \frac{\tilde{\mu}}{\mu}-1 ) A_{\omega}( t ) -\tilde{\mu}( x-\int_0^t{c( s ;\omega ,\mu ) ds} )}h(t;\omega)
\end{aligned}
\end{equation*}
satisfy

\begin{equation*}
\setlength\abovedisplayskip{1pt}
\setlength\belowdisplayskip{1pt}
\left\{
\begin{aligned}
\pa_t \tilde{u}\leq&H\tilde{u}+\tilde{u}(a_1(\theta _t\omega)- b_1(\theta_t\omega)\tilde{u}-c_1(\theta _t\omega)(v^*(t;\omega)-\tilde{v})) ,\\
\pa_t \tilde{v}\leq&H\tilde{v}+b_2(\theta _t\omega)(v^*(t;\omega)-\tilde{v})\tilde{u}+\tilde{v}(a_2(\theta _t\omega)- 2c_2(\theta _t\omega)v^*(t;\omega)+c_2(\theta _t\omega)\tilde{v}),
\end{aligned}
\right.
\end{equation*}
for $ t\in ( t_k,t_{k+1} )$, $ x\ge \int_0^t{c( s ;\omega ,\mu ) ds}+\frac{\ln d}{\tilde{\mu}-\mu}+\frac{A_{\omega}\left( t \right)}{\mu} $, $ k\in \mathbb{Z} $.
 \end{lemma}

 \begin{proof}
For given $\omega\in\Omega_0$ and $ 0<\mu <\tilde{\mu}<\min \{  2\mu, {\mu }^* \} $, by the arguments in the proof of \cite[Lemma 5.1]{CaSh17} we can get that $$\frac{e^{\tilde{\mu}}+e^{-\tilde{\mu}}-2+\underline{\lambda}}{\tilde{\mu}}<\frac{e^{\mu}+e^{-\mu}-2+\underline{\lambda}}{\mu} ,$$ and hence $$ \underline{\lambda} >\frac{\mu ( e^{\tilde{\mu}}+e^{-\tilde{\mu}}-2 ) -\tilde{\mu}( e^{\mu}+e^{-\mu}-2 )}{\tilde{\mu}-\mu} .$$ Let $ 0<\delta \ll 1 $ be such that $$ ( 1-\delta )\underline{\lambda} >\frac{\mu ( e^{\tilde{\mu}}+e^{-\tilde{\mu}}-2 ) -\tilde{\mu}( e^{\mu}+e^{-\mu}-2 )}{\tilde{\mu}-\mu} .$$
It then follows from  Lemma \ref{technical} that there exist $ T>0 $ and $ A_{\omega}\in W_{loc}^{1,\infty}( \mathbb{R} ) \cap L^{\infty}( \mathbb{R} )$ such that $ A_{\omega}( \cdot ) \in C^1(( t_k,t_{k+1}))  $ with $ t_k=kT $ for $ k\in \mathbb{Z} $, and \begin{equation}\label{inequality-A(t)}
\setlength\abovedisplayskip{1pt}
\setlength\belowdisplayskip{1pt}
( 1-\delta ) (a_1( \theta _t\omega )-c_1(\theta _t\omega)v^*(t;\omega)) +A_{\omega}^{\prime}( t)\ge  \frac{\mu ( e^{\tilde{\mu}}+e^{-\tilde{\mu}}-2 ) -\tilde{\mu}( e^{\mu}+e^{-\mu}-2 )}{\tilde{\mu}-\mu}
\end{equation}
for all $ t\in ( t_k,t_{k+1})  $, $ k\in \mathbb{Z} $.

Now fix $\delta>0$ and $A_{\omega}( t )$ chosen in the above inequality.

Let $$ \xi ( x,t;\omega )=x-\int_0^t{c( s ;\omega ,\mu) ds} ,$$ $$ \tilde{u}( x,t, \omega)=e^{-\mu   \xi ( x,t;\omega ) }-de^{( \frac{\tilde{\mu}}{\mu}-1 ) A_{\omega}( t ) -\tilde{\mu}   \xi ( x,t;\omega  )} ,$$
and
$$ \tilde{v}( x,t, \omega)=\sigma e^{-\mu   \xi ( x,t;\omega ) }h(t;\omega)-\sigma de^{( \frac{\tilde{\mu}}{\mu}-1 ) A_{\omega}( t ) -\tilde{\mu}   \xi ( x,t;\omega  )}h(t;\omega)$$ with $ d>1 $ to be determined later.
Recall that $$c(t;\omega, \mu)=\frac{e^{\mu}+e^{-\mu}-2+ a_1( \theta _t\omega) -c_1( \theta _t\omega)v^*(t;\omega )}{\mu}.$$
\vspace{-1cm}
Then we have

\begin{align}
\label{more1}
\partial_t& \tilde{u}-[H\tilde{u}+\tilde{u}(a_1(\theta _t\omega)- b_1(\theta_t\omega)\tilde{u}-c_1(\theta _t\omega)(v^*(t;\omega)-\tilde{v}))]\nonumber\\
=&\mu c( t;\omega,\mu ) e^{-\mu \xi ( x,t;\omega )}+d[ -( \frac{\tilde{\mu}}{\mu}-1 ) A_{\omega}^{\prime}( t ) -\tilde{\mu}c( t;\omega,\mu ) ] e^{( \frac{\tilde{\mu}}{\mu}-1 ) A_{\omega}( t ) -\tilde{\mu}\xi ( x,t;\omega )}\nonumber\\
&-[( e^{\mu}+e^{-\mu}-2 ) e^{-\mu \xi ( x,t;\omega )}-d( e^{\tilde{\mu}}+e^{-\tilde{\mu}}-2 ) e^{( \frac{\tilde{\mu}}{\mu}-1 ) A_{\omega}( t ) -\tilde{\mu}\xi ( x,t;\omega)}] \nonumber\\
&-\tilde{u}[a_1(\theta _t\omega)- b_1(\theta_t\omega)\tilde{u}-c_1(\theta _t\omega)(v^*(t;\omega)-\tilde{v})]\nonumber\\
 =& d[ -( \frac{\tilde{\mu}}{\mu}-1 ) A_{\omega}^{\prime}( t ) -\tilde{\mu}c( t;\omega,\mu ) + e^{\tilde{\mu}}+e^{-\tilde{\mu}}-2 +a_1( \theta _t\omega )-c_1( \theta _t\omega )v^*(t;\omega) ]\nonumber\\
 &\times e^{( \frac{\tilde{\mu}}{\mu}-1 ) A_{\omega}( t ) -\tilde{\mu}\xi ( x,t;\omega )}+\tilde{ u}[b_1( \theta _t\omega )\tilde{ u}-c_1( \theta _t\omega )\tilde{ v}] \nonumber\\
=&d( \frac{\tilde{\mu}}{\mu}-1 ) [ \frac{\mu ( e^{\tilde{\mu}}+e^{-\tilde{\mu}}-2 ) -\tilde{\mu}( e^{\mu}+e^{-\mu}-2 )}{\tilde{\mu}-\mu}-( 1-\delta )(a_1( \theta _t\omega )-c_1( \theta _t\omega )v^*(t;\omega))-A_{\omega}^{\prime}( t ) ]\nonumber\\
&\times e^{( \frac{\tilde{\mu}}{\mu}-1 ) A_{\omega}( t ) -\tilde{\mu}\xi ( x,t;\omega )}+ \tilde{ u}[b_1( \theta _t\omega )\tilde{ u}-c_1( \theta _t\omega )\sigma h(t;\omega)\tilde{ u}]\nonumber\\
&-\delta d( \frac{\tilde{\mu}}{\mu}-1 )(a_1( \theta _t\omega )-c_1( \theta _t\omega )v^*(t;\omega)) e^{( \frac{\tilde{\mu}}{\mu}-1 ) A_{\omega}( t ) -\tilde{\mu}\xi ( x,t;\omega )}\nonumber\\
\leq &d( \frac{\tilde{\mu}}{\mu}-1 ) [ \frac{\mu ( e^{\tilde{\mu}}+e^{-\tilde{\mu}}-2 ) -\tilde{\mu}( e^{\mu}+e^{-\mu}-2 )}{\tilde{\mu}-\mu}-( 1-\delta )(a_1( \theta _t\omega )-c_1( \theta _t\omega )v^*(t;\omega))-A_{\omega}^{\prime}( t ) ]\nonumber\\
&\times e^{( \frac{\tilde{\mu}}{\mu}-1 ) A_{\omega}( t ) -\tilde{\mu}\xi ( x,t;\omega )}+b_1(\theta_t\omega)\tilde{u}^2-\delta d( \frac{\tilde{\mu}}{\mu}-1 )(a_1( \theta _t\omega )-c_1( \theta _t\omega )v^*(t;\omega)) e^{( \frac{\tilde{\mu}}{\mu}-1 ) A_{\omega}( t ) -\tilde{\mu}\xi ( x,t;\omega )}\nonumber\\
=&d( \frac{\tilde{\mu}}{\mu}-1 ) [ \frac{\mu ( e^{\tilde{\mu}}+e^{-\tilde{\mu}}-2 ) -\tilde{\mu}( e^{\mu}+e^{-\mu}-2 )}{\tilde{\mu}-\mu}-( 1-\delta )(a_1( \theta _t\omega )-c_1( \theta _t\omega )v^*(t;\omega))-A_{\omega}^{\prime}( t ) ]\nonumber\\
&\times e^{( \frac{\tilde{\mu}}{\mu}-1 ) A_{\omega}( t ) -\tilde{\mu}\xi ( x,t;\omega)}-[d\delta( \frac{\tilde{\mu}}{\mu}-1 )e^{( \frac{\tilde{\mu}}{\mu}-1 ) A_{\omega}( t )}(a_1( \theta _t\omega )
-c_1( \theta _t\omega )v^*(t;\omega))\nonumber\\
&-b_1(\theta_t\omega)e^{-(2\mu-\tilde{\mu})\xi ( x,t;\omega)}] e^{-\tilde{\mu}\xi( x,t;\omega )}+d[ -2e^{-\mu \xi ( x,t;\omega )}+de^{( \frac{\tilde{\mu}}{\mu}-1 ) A_{\omega}( t ) -\tilde{\mu}\xi ( x,t;\omega )} ]\nonumber\\ &\times e^{( \frac{\tilde{\mu}}{\mu}-1 ) A_{\omega}( t ) -\tilde{\mu}\xi ( x,t;\omega )}b_1( \theta _t\omega )
\end{align}
for $ t\in ( t_k,t_{k+1} )  $. Note that
\begin{align*}
\pa_t & \tilde{v}-[H\tilde{v}+b_2(\theta _t\omega)(v^*(t;\omega)-\tilde{v})\tilde{u}+\tilde{v}(a_2(\theta _t\omega)- 2c_2(\theta _t\omega)v^*(t;\omega)+c_2(\theta _t\omega)\tilde{v})]\nonumber\\
=&\sigma [ ( a_2( \theta _t\omega) -2c_2( \theta _t\omega ) v^*( t;\omega ) ) h( t;\omega ) +b_2( \theta _t\omega ) v^*( t;\omega ) -( a_1( \theta _t\omega ) -c_1( \theta _t\omega ) v^*( t;\omega ) ) h( t;\omega ) ] \tilde{u}\nonumber\\
&+\sigma h( t;\omega ) \pa_t \tilde{u}-\sigma h( t;\omega ) H\tilde{u}-b_2( \theta _t\omega ) v^*( t;\omega ) \tilde{u}+\sigma h( t;\omega ) \tilde{u}^2b_2( \theta _t\omega )\nonumber\\
&-\sigma h( t;\omega ) \tilde{u}( a_2( \theta _t\omega ) -2c_2( \theta _t\omega ) v^*( t;\omega ) +c_2( \theta _t\omega ) \tilde{v} )\nonumber\\
=&\sigma h( t;\omega ) \{ \partial _t\tilde{u}-[ H\tilde{u}+\tilde{u}( a_1( \theta _t\omega ) -b_2( \theta _t\omega ) \tilde{u}-c_1( \theta _t\omega ) v^*( t;\omega ) +c_2( \theta _t\omega ) \tilde{v} ) ]\} \nonumber\\
&+b_2( \theta _t\omega ) v^*( t;\omega ) \tilde{u}( \sigma -1 ).\nonumber\\
\end{align*}
Then by similar arguments as proving \eqref{more1},  we can get
\begin{align}
\label{more2}
\pa_t & \tilde{v}-[H\tilde{v}+b_2(\theta _t\omega)(v^*(t;\omega)-\tilde{v})\tilde{u}+\tilde{v}(a_2(\theta _t\omega)- 2c_2(\theta _t\omega)v^*(t;\omega)+c_2(\theta _t\omega)\tilde{v})]\nonumber\\
\leq & \{d( \frac{\tilde{\mu}}{\mu}-1 ) [ \frac{\mu ( e^{\tilde{\mu}}+e^{-\tilde{\mu}}-2 ) -\tilde{\mu}( e^{\mu}+e^{-\mu}-2 )}{\tilde{\mu}-\mu}-( 1-\delta )(a_1( \theta _t\omega )-c_1( \theta _t\omega )v^*(t;\omega))\nonumber\\
&-A_{\omega}^{\prime}( t ) ]e^{( \frac{\tilde{\mu}}{\mu}-1 ) A_{\omega}( t ) -\tilde{\mu}\xi ( x,t;\omega)}+[b_2(\theta_t\omega)e^{-(2\mu-\tilde{\mu})\xi ( x,t;\omega)}-d\delta( \frac{\tilde{\mu}}{\mu}-1 )e^{( \frac{\tilde{\mu}}{\mu}-1 ) A_{\omega}( t )}\nonumber\\
&\times (a_1( \theta _t\omega )-c_1( \theta _t\omega )v^*(t;\omega))]e^{-\tilde{\mu}\xi( x,t;\omega )}+d[ -2e^{-\mu \xi ( x,t;\omega )}+de^{( \frac{\tilde{\mu}}{\mu}-1 ) A_{\omega}( t ) -\tilde{\mu}\xi ( x,t;\omega )} ] \nonumber\\
&\times e^{( \frac{\tilde{\mu}}{\mu}-1 ) A_{\omega}( t ) -\tilde{\mu}\xi ( x,t;\omega )}b_2( \theta _t\omega )\}\sigma h( t;\omega ) +b_2( \theta _t\omega ) v^*( t;\omega ) \tilde{u}( \sigma -1 )
\end{align}
for $ t\in ( t_k,t_{k+1} )  $.

Let
$$d\ge d_{\omega}:= \max \{ \max\limits_{i\in\{1,2\},t\in\R} \{\frac{b_i( \theta _t\omega )}{a_1( \theta _t\omega ) -c_1( \theta _t\omega ) v^*( t;\omega )}\}\frac{\mu e^{-( \frac{\tilde{\mu}}{\mu}-1 )  \lVert A_{\omega} \rVert _{\infty}}}{\delta ( \tilde{\mu}-\mu )} , \ e^{( \frac{\tilde{\mu}}{\mu}-1 ) \lVert A_{\omega} \rVert _{\infty}} \}.$$
Then we have
$$ d\delta ( \frac{\tilde{\mu}}{\mu}-1 ) e^{( \frac{\tilde{\mu}}{\mu}-1 ) A_{\omega}( t )}(a_1( \theta _t\omega ) -c_1( \theta _t\omega ) v^*( t;\omega ))\ge b_i(\theta_t\omega) \ \  (i=1,2).$$ For this choice of $d$, if $\xi ( x,t;\omega ) =x-\int_0^t{c(s ;\omega ,\mu ) ds}\geq\frac{\ln d}{\tilde{\mu}-\mu}+\frac{A_{\omega}( t )}{\mu} $, which is equivalent to $\tilde{u}( x,t,\omega)\ge 0$ and $\tilde{v}( x,t,\omega)\ge 0$, then $\xi ( x,t;\omega )\ge 0$ and $de^{( \frac{\tilde{\mu}}{\mu}-1 ) A_{\omega}( t ) -\tilde{\mu}\xi ( x,t;\omega )}\leq e^{-\mu \xi ( x,t;\omega )}$. Together with \eqref{inequality-A(t)}, we get that each term the right hand side of (\ref{more1}) and (\ref{more2}) is less than or equal to zero. The lemma thus follows.
\end{proof}

For given function  $t\mapsto u(t)\in l^\infty(\Z)$ and $c\in\R$, we define
$$\limsup_{|i|\le ct, t\to\infty} u_i(t)=\limsup_{t\to\infty}\sup_{i\in\Z,|i|\le ct} u_i(t).$$

\begin{lemma}
\label{speed}
Let $(u^0,v^0)\in l^{\infty,+}(\Z)\times l^{\infty,+}(\Z)$. If there is a positive constant $c(\omega)>0$ such that
\begin{align}
\label{speed*}
\liminf_{s\in\R,|i|\le c(\omega)t, t\to\infty} u_i(t;u^0,v^0,\theta_s\omega)=\liminf_{t\to\infty}\inf_{s\in\R, i\in\Z,|i|\le c(\omega)t}u_i(t;u^0,v^0,\theta_s\omega) >0,
\end{align}
then for any $0<c<c(\omega)$,
\begin{equation}
\label{convergence-aux-eq1}
\limsup_{|i|\le ct,t\to\infty}[ |u_i(t;u^0,v^0,\theta_s\omega)-u^*(t+s;\omega)|+|v_i(t;u^0,v^0,\theta_s\omega)-v^*(t+s;\omega)|]=0
\end{equation}
uniformly in $s\in\R$.
\end{lemma}

\begin{proof}
Let $\omega\in\Omega_0$ and $c(\omega)$ satisfy \eqref{speed*}. Denote
$$
\delta_0=\liminf_{s\in\R,|i|\le c(\omega)t, t\to\infty} u_i(t;u^0,v^0,\theta_s\omega).
$$
Then there is $T\gg1$ such that
\begin{align}
\label{value1}
\inf_{ |i|\le c(\omega)t} u_i(t;u^0,v^0,\theta_s\omega)\ge \frac{\delta_0}{2},\quad \forall\,\, s\in\R,\,\, t\ge T.
\end{align}
Suppose by contradiction that there is $0<c_0<c(\omega)$ such that \eqref{convergence-aux-eq1} does not hold. Then there are $\epsilon_0>0$,
$s_n\in\R$, $i_n\in\Z$, $t_n>0$ such that $|i_n|\le c_0 t_n$, $t_n\to\infty$, and
\begin{equation}
\label{convergence-aux-eq3}
|u_{i_n}(t_n;u^0,v^0,\theta_{s_n}\omega)-u^*(t_n+s_n;\omega)|+|v_{i_n}(t_n;u^0,v^0,\theta_{s_n}\omega)-v^*(t_n+s_n;\omega)|\ge\epsilon_0.
\end{equation}
Let $(\tilde  u^0,\tilde  v^0)=\{(\tilde u^0_i,\tilde v^0_i)\}$ and $(\hat u^0,\hat v^0)=\{(\hat u^0_i,\hat v^0_i)\}$, where $\tilde u^0_i=\frac{\delta_0}{2}$, $\tilde v^0_i=0$, $\hat u^0_i=\|u^0\|$ and $\hat v^0_i=\|v^0\|$ for all $i\in\Z$. By the global stability of $(u^*(t;\omega),v^*(t;\omega))$, there is $\tilde T\ge T$ such that
\begin{equation}
\label{convergence-aux-eq4}
|u_i(t;\tilde{u}^0,\tilde{v}^0,\theta_{s_n}\omega)-u^*(t+s;\omega)|+|v_i(t;\tilde{u}^0,\tilde{v}^0,\theta_{s_n}\omega)-v^*(t+s;\omega)|<\frac{\epsilon_0}{4}
\end{equation}  for all  $i\in\Z, s\in\R, t\ge \tilde T$, 
and
\begin{align}
\label{convergence-aux-eq4-1}
u_i(t;u^0,v^0,\theta_s\omega)&\le u_i(t;\hat u^0,\hat v^0,\theta_s\omega)< u^*(t+s;\omega)+\frac{\epsilon_0}{2},\nonumber\\
v_i(t;u^0,v^0,\theta_s\omega)&\le v_i(t;\hat u^0,\hat v^0,\theta_s\omega)< v^*(t+s;\omega)+\frac{\epsilon_0}{2},\quad \forall\,\, i\in\Z,\,\, s\in\R,\,\, t\ge \tilde T.
\end{align}
Observe that $(c(\omega ) -c_0)( t_n-\tilde{T}) -2c_0\tilde{T}\rightarrow \infty \ as\ n\rightarrow \infty$. Hence there is $N$ such that
$$(c(\omega ) -c_0)( t_n-\tilde{T}) -2c_0\tilde{T}\geq T,\ \ \ \forall n\geq N.$$
For every $n\geq N$, let $\tilde u^n=\{\tilde u^n_i\}\in l^\infty(\Z)$ with $\|\tilde{u}^n \|\leq \frac{\delta_0}{2}$ and
\begin{align}
\label{value2}
\tilde u^n_i&=\begin{cases} \frac{\delta_0}{2},\ \ \  |i|\le (c(\omega ) -c_0)( t_n-\tilde{T}) -2c_0\tilde{T},\cr
0,\ \ \ \ |i|\geq (c(\omega ) -c_0)( t_n-\tilde{T}) -c_0\tilde{T},
\end{cases} \nonumber\\
\tilde{v}^n&\equiv0.
\end{align}
Since $|i|\le (c(\omega ) -c_0)( t_n-\tilde{T}) -c_0\tilde{T}$ implies that $$|i+i_n|\leq c(\omega)(t_n-\tilde{T}) \ \ \mathrm{for} \ \mathrm{every} \ n\geq N,$$ it follows from \eqref{value1} and \eqref{value2} that
$$\tilde{u}^n_i\leq u_{i+i_n}(t_n-\tilde{T};u^0,v^0,\theta_{s_n}\omega), \ \forall i\in\Z,\ \forall n\geq N.$$
Note that
$$\tilde{v}^n_i=0\leq v_{i+i_n}(t_n-\tilde{T};u^0,v^0,\theta_{s_n}\omega),\ \forall i\in\Z,\ \forall n\geq N.$$
Then by comparison principle, we have
\begin{align}
\label{value3}
u_i( t;\tilde{u}^n,\tilde{v}^n,\theta _{\tilde{s}_n}\omega) \le u_{i+i_n}( t+t_n-\tilde{T};u^0,v^0,\theta _{s_n}\omega ) ,\ \ \forall i\in \mathbb{Z},\ t>0,\ n\ge N,
\end{align}
and
\begin{align}
\label{value4}
v_i( t;\tilde{u}^n,\tilde{v}^n,\theta _{\tilde{s}_n}\omega) \le v_{i+i_n}( t+t_n-\tilde{T};u^0,v^0,\theta _{s_n}\omega ) ,\ \ \forall i\in \mathbb{Z},\ t>0,\ n\ge N,
\end{align}
where $\tilde{s}_n=s_n+t_n-\tilde{T}$. It follows from the definition of $(\tilde{u}^n,\tilde{v}^n)$ that $$\lim_{n\to\infty}(\tilde{u}^n,\tilde{v}^n)=(\tilde{u}^0,\tilde{v}^0)\ \ \mathrm{locally}\ \mathrm{uniformly}\ \mathrm{in}\  i\in\Z.$$
Therefore, from Proposition \ref{convergence-lemma} we have that for every $t>0$,
\begin{align}
\label{value5}
&\lim_{n\rightarrow\infty}[|u_i( t;\tilde{u}^n,\tilde{v}^n,\theta _{\tilde{s}_n}\omega)-u_i( t;\tilde{u}^0,\tilde{v}^0,\theta _{\tilde{s}_n}\omega)|\nonumber\\ &\ \ \ \  \ \ \  +|v_i( t;\tilde{u}^n,\tilde{v}^n,\theta _{\tilde{s}_n}\omega)-v_i( t;\tilde{u}^0,\tilde{v}^0,\theta _{\tilde{s}_n}\omega)|]=0
\end{align}
locally uniformly in $i\in\Z$.
It then follows from \eqref{convergence-aux-eq4}, \eqref{value3}, \eqref{value4} and \eqref{value5} that
\begin{align*}
u^*( s_n+t_n;\omega ) -\frac{\epsilon _0}{2}<u_0( \tilde{T};\tilde{u}^n,\tilde{v}^n,\theta _{\tilde{s}_n}\omega ) \le u_{i_n}( t_n;u^0,v^0,\theta _{s_n}\omega )
\end{align*}
and
\begin{align*}
v^*( s_n+t_n;\omega ) -\frac{\epsilon _0}{2}<v_0( \tilde{T};\tilde{u}^n,\tilde{v}^n,\theta _{\tilde{s}_n}\omega ) \le v_{i_n}( t_n;u^0,v^0,\theta _{s_n}\omega )\ \ \ \mathrm{for}\ \ n\gg1.
\end{align*}
Note that by \eqref{convergence-aux-eq4-1} we have
\begin{align*}
u_{i_n}( t_n;u^0,v^0,\theta _{s_n}\omega )<u^*( s_n+t_n;\omega ) +\frac{\epsilon _0}{2}
\end{align*}
and
\begin{align*}
v_{i_n}( t_n;u^0,v^0,\theta _{s_n}\omega )<v^*( s_n+t_n;\omega ) +\frac{\epsilon _0}{2}\ \ \ \mathrm{for}\ \ n\gg1.
\end{align*}
Then
$$|u_{i_n}( t_n;u^0,v^0,\theta _{s_n}\omega )- u^*( s_n+t_n;\omega )|+|v_{i_n}( t_n;u^0,v^0,\theta _{s_n}\omega )- v^*( s_n+t_n;\omega )|<\epsilon_0$$
for $n\gg1$, which contradicts to \eqref{convergence-aux-eq3}. Hence \eqref{convergence-aux-eq1} holds.
\end{proof}

\medskip

\section{Random transition fronts}

In this section, we study the existence and non-existence of random transition fronts, and prove
Theorem \ref{exist-thm}.

For any $\gamma>c_0$, let $0<\mu<\mu^*$ be such that $\frac{e^{\mu }+e^{-\mu}-2+\underline{\lambda}}{\mu}=\gamma $, where  $\underline{\lambda} =\underline{a_1(\omega ) -c_1( \omega )v^*( \cdot;\omega )}$ for $\omega\in\Omega_0$. For every $\omega\in\Omega$, denote  $c(t;\omega, \mu)=\frac{e^{\mu}+e^{-\mu}-2+(a_1( \theta _t\omega ) -c_1( \theta _t\omega) v^*( t;\omega ) )}{\mu}$ and $ \hat{u}^{\mu}( x,t;\omega ) =e^{-\mu ( x-\int_0^t{c( s ;\omega ,\mu ) ds} )}$.
Then $ \hat{u}^{\mu}( x,t;\omega)$ satisfies
\begin{align*}
\partial &_t\hat{u}^{\mu}( x,t;\omega ) -H\hat{u}^{\mu}( x,t;\omega ) -(a_1( \theta _t\omega ) -c_1( \theta _t\omega) v^*( t;\omega ) ) \hat{u}^{\mu}( x,t;\omega )\\
=&\hat{u}^{\mu}( x,t;\omega ) [ \mu c( t;\omega ,\mu ) -( e^{\mu}+e^{-\mu}-2)+(a_1( \theta _t\omega ) -c_1( \theta _t\omega) v^*( t;\omega ) ) ]=0,\ \ \text{for\ }x\in \mathbb{R},\ t\in \mathbb{R}.
\end{align*}
Then we have that
\begin{equation}
\begin{aligned}
\pa_t &\hat{u}^\mu-H\hat{u}^\mu-\hat{u}^\mu(a_1(\theta _t\omega)- b_1(\theta_t\omega)\hat{u}^\mu-c_1(\theta _t\omega)(v^*(t;\omega)-\hat{u}^\mu))\nonumber\\
=&\hat{u}^\mu[\mu c( t;\omega ,\mu )-( e^{\mu}+e^{-\mu}-2)-(a_1( \theta _t\omega ) -c_1( \theta _t\omega) v^*( t;\omega ))]+\hat{u}^\mu(b_1( \theta _t\omega ) -c_1( \theta _t\omega))\hat{u}^\mu\nonumber\\
=&\hat{u}^\mu (b_1( \theta _t\omega ) -c_1( \theta _t\omega))\hat{u}^\mu\nonumber\\
\geq&0,\nonumber\\
\end{aligned}
\end{equation}
and
\begin{equation}
\begin{aligned}
\pa_t &\hat{u}^\mu-H\hat{u}^\mu-b_2(\theta _t\omega)(v^*(t;\omega)-\hat{u}^\mu)\hat{u}^\mu-\hat{u}^\mu(a_2(\theta _t\omega)- 2c_2(\theta _t\omega)v^*(t;\omega)+c_2(\theta _t\omega)\hat{u}^\mu)\nonumber\\
=&\mu c( t;\omega ,\mu )\hat{u}^\mu-( e^{\mu}+e^{-\mu}-2)\hat{u}^\mu-b_2(\theta _t\omega)v^*(t;\omega)\hat{u}^\mu+\hat{u}^\mu b_2(\theta _t\omega)\hat{u}^\mu\nonumber\\
&-(a_2(\theta _t\omega)- 2c_2(\theta _t\omega)v^*(t;\omega))\hat{u}^\mu-\hat{u}^\mu c_2(\theta _t\omega)\hat{u}^\mu\nonumber\\
=&[a_1( \theta _t\omega ) -c_1( \theta _t\omega ) v^*( t;\omega )-(a_2( \theta _t\omega ) -2c_2( \theta _t\omega ) v^*( t;\omega )+b_2(\theta_t\omega)v^*(t;\omega))]\hat{u}^\mu\nonumber\\
&+\hat{u}^\mu (b_2( \theta _t\omega )-c_2( \theta _t\omega ))\hat{u}^\mu\nonumber\\
\geq&0\ \ \  \mathrm{for}\ \ x\in \mathbb{R},\ t\in \mathbb{R}.\nonumber\\
\end{aligned}
\end{equation}
Hence, $ (\hat{u}^{\mu}( x,t;\omega ),\hat{u}^{\mu}( x,t;\omega )) =(e^{-\mu ( x-\int_0^t{c( s ;\omega ,\mu ) ds} )},e^{-\mu ( x-\int_0^t{c( s ;\omega ,\mu ) ds} )})$ is a super-solution of (\ref{main-trans1-continuous}). Denote
$$(\overline{u}^{\mu}( x,t;\omega ) , \overline{v}^{\mu}( x,t;\omega )) =\min \{(u^*(t;\omega),v^*(t;\omega)),(\hat{u}^{\mu}( x,t;\omega ),\hat{u}^{\mu}( x,t;\omega ))\}. $$ Then $(\overline{u}^{\mu}( x,t;\omega ) , \overline{v}^{\mu}( x,t;\omega ))$ is a generalized super-solution of \eqref{main-trans1-continuous}.

\begin{lemma}\label{monotone-sup}
For $\omega\in\Omega_0$, we have that
\begin{align*}
 &u( x,t-t_0;\overline{u}^{\mu}( \cdot ,t_0;\omega),\overline{v}^{\mu}( \cdot ,t_0;\omega) ,\theta_{t_0}\omega) \le \overline{u}^{\mu}( x,t;\omega )\  \ \ and \\
 \ \ \ \  &v( x,t-t_0;\overline{u}^{\mu}( \cdot ,t_0;\omega),\overline{v}^{\mu}( \cdot ,t_0;\omega) ,\theta_{t_0}\omega) \le \overline{v}^{\mu}( x,t;\omega ) ,\ \ \forall x\in \mathbb{R},\ t\geq t_0,\  t_0\in\R.
\end{align*}
\end{lemma}

\begin{proof}
For any constant $C$, $ (\hat{U}( x,t;\omega),\hat{V}( x,t;\omega)) :=(e^{Ct}\hat{u}^{\mu}( x,t;\omega),e^{Ct}\hat{u}^{\mu}( x,t;\omega))$ satisfies
\begin{align*}
\partial _t\hat{U}( x,t;\omega) &=( \partial _t\hat{u}^\mu( x,t;\omega) +C\hat{u}^{\mu}( x,t;\omega)) e^{Ct}\\
&\ge H\hat{U}( x,t;\omega) +C\hat{U}( x,t;\omega) +e^{Ct}f(t,\hat{u},\hat{u},\omega),
\end{align*}
and
\begin{align*}
\partial _t\hat{V}( x,t;\omega) &=( \partial _t\hat{u}^\mu( x,t;\omega) +C\hat{u}^{\mu}( x,t;\omega)) e^{Ct}\\
&\ge H\hat{V}( x,t;\omega) +C\hat{V}( x,t;\omega) +e^{Ct}g(t,\hat{u},\hat{u},\omega).
\end{align*}
Hence,
$$
\hat{U}( x,t;\omega) \ge \hat{U}( x,t_0;\omega) +\int_{t_0}^t{( H\hat{U}( x,\tau ;\omega) +C\hat{U}( x,\tau ;\omega) +e^{C\tau}f(\tau,\hat{u},\hat{u},\omega)) d\tau},
$$
and
$$
\hat{V}( x,t;\omega) \ge \hat{V}( x,t_0;\omega) +\int_{t_0}^t{( H\hat{V}( x,\tau ;\omega) +C\hat{V}( x,\tau ;\omega) +e^{C\tau}g(\tau,\hat{u},\hat{u},\omega)) d\tau}.
$$
Denote $ (\overline{U}( x,t;\omega),\overline{V}( x,t;\omega)) :=(e^{Ct}\overline{u}^{\mu}( x,t;\omega),e^{Ct}\overline{v}^{\mu}( x,t;\omega))$. Then we also have
$$
\overline{U}( x,t;\omega) \ge \overline{U}( x,t_0;\omega) +\int_{t_0}^t{( H\overline{U}( x,\tau ;\omega ) +C\overline{U}( x,\tau ;\omega) +e^{C\tau}f(\tau,\overline{u},\overline{v},\omega)) d\tau},
$$
and
$$
\overline{V}( x,t;\omega) \ge \overline{V}( x,t_0;\omega) +\int_{t_0}^t{( H\overline{V}( x,\tau ;\omega ) +C\overline{V}( x,\tau ;\omega) +e^{C\tau}g(\tau,\overline{u},\overline{v},\omega)) d\tau}.
$$

Let $ Q_1( x,t;\omega) =e^{Ct}( \overline{u}^{\mu}( x,t;\omega) -u( x,t-t_0;\overline{u}^{\mu}( \cdot ,t_0;\omega),\overline{v}^{\mu}( \cdot ,t_0;\omega),\theta_{t_0}\omega))$ and $ Q_2( x,t;\omega) =e^{Ct}( \overline{v}^{\mu}( x,t;\omega) -v( x,t-t_0;\overline{u}^{\mu}( \cdot ,t_0;\omega),\overline{v}^{\mu}( \cdot ,t_0;\omega),\theta_{t_0}\omega))$. Then
\begin{align*} &Q_1( x,t;\omega) \\ &\ge Q_1( x,t_0;\omega) +\int_{t_0}^t{( HQ_1( x,\tau ;\omega) +a_1( x,\tau ;\omega)Q_1( x,\tau ;\omega)+b_1( x,\tau ;\omega)Q_2( x,\tau ;\omega) ) d\tau}, \end{align*}
and
\begin{align*} &Q_2( x,t;\omega) \\ &\ge Q_2( x,t_0;\omega) +\int_{t_0}^t{( HQ_2( x,\tau ;\omega) +a_2( x,\tau ;\omega)Q_1( x,\tau ;\omega)+b_2( x,\tau ;\omega)Q_2( x,\tau ;\omega) ) d\tau}, \end{align*}
where
\begin{align*}
 a_1( x,t;\omega)&=C+f_u(t,u^*_1,v^*_1,\omega),\ \   b_1( x,t;\omega)=f_v(t,u^*_1,v^*_1,\omega),\\
 a_2( x,t;\omega)&=g_u(t,u^*_2,v^*_2,\omega),\ \   b_2( x,t;\omega)=C+g_v(t,u^*_2,v^*_2,\omega).
\end{align*}
Since system \eqref{main-trans1-continuous} is cooperative, we know that $ b_1( x,t;\omega)\geq0$ and $ a_2( x,t;\omega)\geq0$. By the boundedness of $\bar{u}^\mu( x,t;\omega)$, $\bar{v}^\mu( x,t;\omega)$, $u(x,t-t_0;\overline{u}^{\mu}( \cdot ,t_0;\omega),\overline{v}^{\mu}( \cdot ,t_0;\omega) ,\theta_{t_0}\omega)$ and $v(x,t-t_0;\overline{u}^{\mu}( \cdot ,t_0;\omega),$ $\overline{v}^{\mu}( \cdot ,t_0;\omega) ,\theta_{t_0}\omega)$, we can choose $C>0$ such that $ b_2( x,t;\omega)\geq0$ and $ a_1( x,t;\omega)\geq0$ for all $t\ge t_0$, $x\in\mathbb{R}$ and a.e. $\omega\in\Omega$. By the arguments of Proposition \ref{comparison}, we have that
$$ Q_i( x,t;\omega) \ge Q_i( x,t_0;\omega) =0, \ \ \ i=1,2,$$
and hence for $\omega\in\Omega_0$, we have that $u( x,t-t_0;\overline{u}^{\mu}( \cdot ,t_0;\omega),\overline{v}^{\mu}( \cdot ,t_0;\omega) ,\theta_{t_0}\omega) \le \overline{u}^{\mu}( x,t;\omega )$ and $v( x,t-t_0;\overline{u}^{\mu}( \cdot ,t_0;\omega),\overline{v}^{\mu}( \cdot ,t_0;\omega) ,\theta_{t_0}\omega) \le \overline{v}^{\mu}( x,t;\omega ) ,\ \ \forall x\in \mathbb{R},\ t\geq t_0,\  t_0\in\R.$
\end{proof}

Next, we construct a sub-solution of (\ref{main-trans1-continuous}).  Let $\tilde{\mu}>0$ be such that $\mu<\tilde{\mu}<\min\{2\mu,{\mu}^*\}$ and $\omega\in\Omega_0$. Let $A_\omega$ and $d_\omega$ be given by Lemma \ref{sub-solution-lemma}, and let
$$ x_{\omega}( t ) =\int_0^t{c( s;\omega ,\mu )}ds+\frac{{\ln}d_{\omega}+{\ln}\tilde{\mu}-{\ln}\mu}{\tilde{\mu}-\mu}+\frac{A_{\omega}( t )}{\mu}. $$
Recall that $$ \tilde{u}( x,t, \omega) =e^{-\mu ( x-\int_0^t{c( s ;\omega ,\mu ) ds} )}-de^{( \frac{\tilde{\mu}}{\mu}-1 ) A_{\omega}( t ) -\tilde{\mu}( x-\int_0^t{c( s ;\omega ,\mu ) ds} )} $$
and
$$\tilde{v}( x,t, \omega)=\sigma e^{-\mu ( x-\int_0^t{c( s ;\omega ,\mu ) ds} )}h(t;\omega)-\sigma de^{( \frac{\tilde{\mu}}{\mu}-1 ) A_{\omega}( t ) -\tilde{\mu}( x-\int_0^t{c( s ;\omega ,\mu ) ds} )}h(t;\omega)$$
By calculation we have that for any given $t\in\mathbb{R}$,
\begin{align}
\label{supremum}
 &(\tilde{u}( x_{\omega}( t ) ,t,\omega ),\tilde{v}( x_{\omega}( t ) ,t,\omega ) )\nonumber\\
 =&(\sup\limits_{x\in \mathbb{R}}\tilde{u}( x,t,\omega ),\sup\limits_{x\in \mathbb{R}}\tilde{v}( x,t,\omega )
 )\nonumber\\
 =&(e^{-\mu ( \frac{\ln {d_{\omega}}}{\tilde{\mu}-\mu}+\frac{A_{\omega}( t )}{\mu} )}e^{ -\mu \frac{ \ln{\tilde{\mu}}-  \ln{\mu}}{\tilde{\mu}-\mu}}( 1-\frac{\mu}{\tilde{\mu}} ),\sigma h(t;\omega)e^{-\mu ( \frac{\ln {d_{\omega}}}{\tilde{\mu}-\mu}+\frac{A_{\omega}( t )}{\mu} )}e^{ -\mu \frac{ \ln{\tilde{\mu}}-  \ln{\mu}}{\tilde{\mu}-\mu}}( 1-\frac{\mu}{\tilde{\mu}} )) .
 \end{align}
Define
\begin{equation}
\begin{aligned}
&(\underline{u}^{\mu}( x,t;\theta _{t_0}\omega ),\underline{v}^{\mu}( x,t;\theta _{t_0}\omega ) )\nonumber\\
=&\left\{ \begin{array}{l} (\tilde{u}( x,t+t_0,\omega ),\tilde{v}( x,t+t_0,\omega )) ,\ \ \text{if\ }x\ge x_{\omega}( t+t_0 ) ,\\ 	 (\tilde{u}( x_{\omega}( t+t_0 ), t+t_0, \omega ),\tilde{v}( x_{\omega}( t+t_0 ), t+t_0, \omega )) ,\ \ \text{if\ }x\le x_{\omega}( t+t_0 ) .\\ \end{array} \right.
 \end{aligned}
 \end{equation}
It is clear that
\begin{align*} &(0,0)\\& <(\underline{u}^{\mu}( \cdot,t;\theta _{t_0}\omega ),\underline{v}^{\mu}( \cdot,t;\theta _{t_0}\omega )) \\& <(\overline{u}^{\mu}( \cdot,t;\theta _{t_0}\omega ),\overline{v}^{\mu}( \cdot,t;\theta _{t_0}\omega )) \le (u^*(t+t_0;\omega),v^*(t+t_0;\omega)) \end{align*}
for all $t,\ t_0\in \mathbb{R}$, and there exists $\tilde{\sigma}>0$ such that
\begin{align}
\label{super1}
\lim\limits_{x\rightarrow \infty} \sup\limits_{t\in\R, t_0\in \mathbb{R}} \frac{\underline{u}^{\mu}( x,t;\theta _{t_0}\omega )}{\overline{u}^{\mu}( x,t;\theta _{t_0}\omega )}=1 \ \ \
and \ \ \
\lim\limits_{x\rightarrow \infty} \sup\limits_{t\in\R, t_0\in \mathbb{R}} \frac{\underline{v}^{\mu}( x,t;\theta _{t_0}\omega )}{\overline{v}^{\mu}( x,t;\theta _{t_0}\omega )}=\tilde{\sigma}.
\end{align}
Note that by the similar arguments as in Lemma \ref{monotone-sup}, we can prove that
$$ u( x,t-t_0;\underline{u}^{\mu}( \cdot ,t_0; \omega) ,\underline{v}^{\mu}( \cdot ,t_0; \omega) ,\theta _{t_0}\omega ) \ge \underline{u}^{\mu}( x,t;\omega )  $$
and
$$ v( x,t-t_0;\underline{u}^{\mu}( \cdot ,t_0; \omega) ,\underline{v}^{\mu}( \cdot ,t_0; \omega) ,\theta _{t_0}\omega ) \ge \underline{v}^{\mu}( x,t;\omega )  $$
for $x\in\mathbb{R}$, $t\ge t_0$ and a.e. $\omega\in\Omega$.


Now we are in a position to prove the main Theorem.

\begin{proof}[Proof of Theorem \ref{exist-thm}]

 (i) By Lemma \ref{monotone-sup} we have that
$$ u( x,t-t_0;\overline{u}^{\mu}( \cdot ,t_0;\omega ) ,\overline{v}^{\mu}( \cdot ,t_0;\omega ) ,\theta _{t_0}\omega ) \le \overline{u}^{\mu}( x,t;\omega ) $$
It then follows that
$$ u( x,\tau_2-\tau_1;\overline{u}^{\mu}( \cdot ,-\tau_2;\omega) ,\overline{v}^{\mu}( \cdot ,-\tau_2;\omega) ,\theta _{-\tau_2}\omega) \le \overline{u}^{\mu}( x,-\tau_1;\omega) $$ for $x\in\R$ and $\tau_2>\tau_1$.
Then we get that
\begin{align*}
u&( x,t+\tau_1;u( \cdot,\tau_2-\tau_1;\overline{u}^{\mu}( \cdot ,-\tau_2;\omega) ,\overline{v}^{\mu}( \cdot ,-\tau_2;\omega) ,\theta _{-\tau_2}\omega),\\
&\ \ \ \ \ \ \ \ \ \ \ \ \ v( \cdot,\tau_2-\tau_1;\overline{u}^{\mu}( \cdot ,-\tau_2;\omega) ,\overline{v}^{\mu}( \cdot ,-\tau_2;\omega) ,\theta _{-\tau_2}\omega),\theta _{-\tau_1}\omega)\\
\leq &u( x,t+\tau_1;\overline{u}^{\mu}( \cdot ,-\tau_1;\omega),\overline{v}^{\mu}( \cdot ,-\tau_1;\omega),\theta _{-\tau_1}\omega)
\end{align*}
for $x\in \mathbb{R}$, $t\ge -\tau_1$, $\tau_2>\tau_1$, and hence
\begin{align*} u&( x,t+\tau_2 ;\overline{u}^{\mu}( \cdot ,-\tau_2 ;\omega) ,\overline{v}^{\mu}( \cdot ,-\tau_2 ;\omega) ,\theta _{-\tau_2}\omega)\\ \leq &u( x,t+\tau_1 ;\overline{u}^{\mu}( \cdot ,-\tau_1 ;\omega) ,\overline{v}^{\mu}( \cdot ,-\tau_1 ;\omega) ,\theta _{-\tau_1}\omega)\end{align*}
for $x\in \mathbb{R}$, $t\ge -\tau_1$, $\tau_2>\tau_1$.
Therefore $\lim\limits_{\tau\rightarrow\infty} u( x,t+\tau ;\overline{u}^{\mu}( \cdot ,-\tau ;\omega) ,\overline{v}^{\mu}( \cdot ,-\tau ;\omega) ,\theta _{-\tau}\omega)$ exists. Similarly, we can get that $\lim\limits_{\tau\rightarrow\infty} v( x,t+\tau ;\overline{u}^{\mu}( \cdot ,-\tau ;\omega) ,\overline{v}^{\mu}( \cdot ,-\tau ;\omega) ,\theta _{-\tau}\omega)$ exists. Define
\begin{equation}
\label{sationary1}
 U(x,t;\omega):=\underset{\tau\rightarrow\infty}{\lim}u( x,t+\tau ;\overline{u}^{\mu}( \cdot ,-\tau ;\omega) ,\overline{v}^{\mu}( \cdot ,-\tau ;\omega) ,\theta _{-\tau}\omega)
 \end{equation}
and
\begin{equation}
\label{sationary2}
 V(x,t;\omega):=\underset{\tau\rightarrow\infty}{\lim}v( x,t+\tau ;\overline{u}^{\mu}( \cdot ,-\tau ;\omega) ,\overline{v}^{\mu}( \cdot ,-\tau ;\omega) ,\theta _{-\tau}\omega)
 \end{equation}
for $x\in \mathbb{R}$, $t\in\R$, $\omega \in \Omega _0$. Then $ (U(x,t;\omega),V(x,t;\omega))$ is non-increasing in $x\in\R$ and by dominated convergence theorem we know that $ (U(x,t;\omega),V(x,t;\omega))$ is a solution of \eqref{main-trans1-continuous}.

We claim that, for every $\omega\in\Omega_0$,
\begin{align}
\label{uniformly}
&\underset{x\rightarrow -\infty}{\lim}(U( x+\int_0^t{c(s;\omega ,\mu ) ds},t;\omega ),V( x+\int_0^t{c(s;\omega ,\mu ) ds},t;\omega )) \nonumber\\
&=(u^*(t;\omega),v^*(t;\omega))\ \ \ \ \text{uniformly\ in}\ t\in \mathbb{R}.
\end{align}
In fact, fix any $\omega\in\Omega_0$, let
$ \hat{x}_\omega=\frac{\ln d_{\omega}+\ln{\tilde{\mu}}-\ln\mu}{\tilde{\mu}-\mu}-\frac{\lVert A_{\omega} \rVert _{\infty}}{\mu},$
it follows from $ (\underline{u}^{\mu}( x,t;\omega) ,\underline{v}^{\mu}( x,t;\omega)) \le (U( x, t; \omega),V( x, t; \omega)) $, \eqref{supremum} and $\inf\limits_{t\in\R} {h(t;\omega)}>0$ that
$$ 0<( 1-\frac{\mu}{\tilde{\mu}}) e^{-\mu(\frac{\ln d_{\omega}+\ln{\tilde{\mu}}-\ln\mu}{\tilde{\mu}-\mu}+\frac{\lVert A_{\omega} \rVert _{\infty}}{\mu})}\le \inf\limits_{t\in \mathbb{R}}U(\hat{x}_\omega+\int_0^t{c( s ;\omega ,\mu) ds, t; \omega} ),  $$
and
$$0<\sigma\inf\limits_{t\in \mathbb{R}} h(t;\omega)( 1-\frac{\mu}{\tilde{\mu}}) e^{-\mu(\frac{\ln d_{\omega}+\ln{\tilde{\mu}}-\ln\mu}{\tilde{\mu}-\mu}+\frac{\lVert A_{\omega} \rVert _{\infty}}{\mu})}   \leq \inf\limits_{t\in \mathbb{R}}{V(\hat{x}_\omega+\int_0^t{c( s ;\omega ,\mu) ds, t; \omega})}.$$
Let $ (u_0( x),v_0(x)) \equiv (u_0,v_0):=(\inf\limits_{t\in \mathbb{R}}U(\hat{x}_\omega+\int_0^t{c(s ;\omega ,\mu) ds, t; \omega}),\inf\limits_{t\in \mathbb{R}}V(\hat{x}_\omega+\int_0^t{c(s ;\omega ,\mu) ds, t; \omega}))$, and $(\tilde{u}_0(x),\tilde{v}_0(x))$ be uniformly continuous such that $(\tilde{u}_0(x),\tilde{v}_0(x))=(u_0(x),v_0(x))$ for $x<\hat{x}_\omega-1$ and $(\tilde{u}_0(x),\tilde{v}_0(x))=(0,0)$ for $x\geq \hat{x}_\omega$.
Then $\underset{n\rightarrow \infty}{\lim}(\tilde{u}_0( x-n),\tilde{v}_0( x-n)) =(u_0( x),v_0( x))$ locally uniformly in $x\in \mathbb{R}.$
Note that by (H2), we have
$$ \underset{t\rightarrow \infty}{\lim}(u( x,t;u_0,v_0,\theta _{t_0}\omega)-u^*(t+t_0;\omega),v( x,t;u_0,v_0,\theta _{t_0}\omega)-v^*(t+t_0;\omega)) =(0,0)$$
uniformly in $t_0\in \mathbb{R}$ and $x\in \mathbb{R}$.
Then for any $\epsilon>0$, there is $T:=T(\epsilon)>0$ such that $$ u^*(t_0+T;\omega)>u( x,T;u_0,v_0,\theta _{t_0}\omega) >u^*(t_0+T;\omega)-\epsilon, \ \ \forall t_0\in \mathbb{R},\ x\in \mathbb{R}. $$
Therefore, from the definition of $c(t,\omega,\mu)$ we know that, \begin{align*} &u^*(t_0+T;\omega)\\&>u( x+\int_0^T{c( s ;\theta_{t_0}\omega ,\mu) ds},T;u_0,v_0,\theta _{t_0}\omega) >u^*(t_0+T;\omega)-\epsilon, \,\,\,\,\forall t_0\in \mathbb{R},\,\,x\in \mathbb{R}. \end{align*}
By Proposition \ref{convergence-lemma}, there is $N:=N(\epsilon)>1$ such that
\begin{align*} &u^*(t_0+T;\omega)\\ &>u( \int_0^T{c( s ;\theta_{t_0}\omega ,\mu ) ds },T;\tilde{u}_0(\cdot-N),\tilde{v}_0(\cdot-N),\theta _{t_0}\omega) >u^*(t_0+T;\omega)-2\varepsilon, \,\,\,\,\forall t_0\in \mathbb{R}. \end{align*}
That is,
\begin{align*} &u^*(t_0+T;\omega)\\&>u( \int_0^T{c( s ;\theta_{t_0}\omega ,\mu ) ds}-N,T;\tilde{u}_0(\cdot),\tilde{v}_0(\cdot),\theta _{t_0}\omega) >u^*(t_0+T;\omega)-2\varepsilon, \,\,\,\,\forall t_0\in \mathbb{R}. \end{align*}
Note that
$$ U( x+\int_0^{t-T}{c( s;\omega ,\mu) ds},t-T;\omega ) \ge \tilde{u}_0( x), \ \ \forall t\in \mathbb{R},\ x\in \mathbb{R},$$
$$ V( x+\int_0^{t-T}{c( s;\omega ,\mu) ds},t-T;\omega ) \ge \tilde{v}_0( x), \ \ \forall t\in \mathbb{R},\ x\in \mathbb{R},$$
and
$$ \int_0^t{c( s ;\omega ,\mu) ds}=\int_0^T{c(s ;\theta _{t-T}\omega ,\mu) ds}+\int_0^{t-T}{c( s ;\omega ,\mu ) d s}. $$
Then we get
\begin{align*}
\label{uniformly1}
u^*(t;\omega)&>U( x+\int_0^t{c( s ;\omega ,\mu ) ds},t;\omega)
 \\&=u( x+\int_0^T{c( s ;\theta_{t-T}\omega ,\mu ) ds},T;U( \cdot+\int_0^{t-T}{c( s;\omega ,\mu ) ds},t-T;\omega), \\
 & \ \ \ \ \ \ \ \ \ \ \ \ \ \ \ \ \ \ \ \ \ \ \ \ \ \ \ \  \ \ \ \ \ \ \ \ \ \ \  \ \ \ V( \cdot+\int_0^{t-T}{c( s;\omega ,\mu ) ds},t-T;\omega), \theta_{t-T}\omega ) \\&>u^*(t;\omega)-2\epsilon, \ \ \ \ \ \ \forall t\in \mathbb{R},\ x\le -N,
\end{align*}
and hence $\underset{x\rightarrow -\infty}{\lim}U( x+\int_0^t{c(s;\omega ,\mu ) ds},t;\omega )=u^*(t;\omega)$ uniformly in $t\in \mathbb{R}$.
Similarly, we can derive $\underset{x\rightarrow -\infty}{\lim}V( x+\int_0^t{c(s;\omega ,\mu ) ds},t;\omega )=v^*(t;\omega)$ uniformly in $t\in \mathbb{R}$. Thus \eqref{uniformly} follows.

Note that by  \eqref{super1} we have that, for every $\omega\in\Omega_0$,
$$ \lim\limits_{x\rightarrow \infty}\sup\limits_{t\in \mathbb{R}} U( x+\int_0^t{c( s ;\omega ,\mu ) ds},t;\omega )=0\ \ \
\mbox{and}\ \ \
\lim\limits_{x\rightarrow \infty}\sup\limits_{t\in \mathbb{R}} V( x+\int_0^t{c( s ;\omega ,\mu ) ds},t;\omega )=0.$$

Set
$$(\tilde{ \varPhi} (x, t;\omega),\tilde{ \varPsi} (x, t;\omega)) =(U( x+\int_0^t{c( s;\omega ,\mu ) ds,t;\omega} ),V( x+\int_0^t{c( s;\omega ,\mu ) ds,t;\omega} ))$$
and
$$(\varPhi(x,\omega),\varPsi(x,\omega))=(\tilde{ \varPhi} (x, 0;\omega),\tilde{ \varPsi} (x, 0;\omega)) . $$
We now claim that $(\tilde{ \varPhi} (x, t;\omega),\tilde{ \varPsi} (x, t;\omega))$ is stationary
ergodic in $t$, that is,  for a.e. $\omega\in\Omega$,
\begin{equation*}
\label{stationary ergodic}
(\tilde{ \varPhi} ( x,t;\omega ),\tilde{ \varPsi} ( x,t;\omega )) =(\tilde{\varPhi}( x,0;\theta _t\omega ),\tilde{\varPsi}( x,0;\theta _t\omega )).
\end{equation*}
In fact, note that for $\omega\in\Omega$,
\begin{align}
\label{speed1}
&\int_{-\tau}^t{c( s;\omega ,\mu )}ds\nonumber\\
&=\int_{-\tau}^t{\frac{e^{\mu}+e^{-\mu}-2+a_1( \theta _s\omega )-c_1( \theta _s\omega )v^*(s;\omega)}{\mu}ds} \nonumber \\
&=\frac{e^{\mu}+e^{-\mu}-2}{\mu}( t+\tau ) +\int_{-\tau}^t{\frac{a_1( \theta _s\omega )-c_1( \theta _s\omega )v^*(s;\omega)}{\mu}ds}
\end{align}
and
\begin{align}
\label{speed2}
&\int_{-( t+\tau )}^0{c( s;\theta _t\omega ,\mu ) ds}\nonumber\\
&=\int_{-( t+\tau )}^0{\frac{e^{\mu}+e^{-\mu}-2+a_1( \theta _s\circ\theta _t\omega )-c_1( \theta _s\circ\theta _t\omega )v^*(s;\theta_t\omega)}{\mu}ds}\nonumber \\
&=\frac{e^{\mu}+e^{-\mu}-2}{\mu}( t+\tau ) +\int_{-( t+\tau )}^0{\frac{a_1( \theta _{s+t}\omega )-c_1( \theta _{s+t}\omega )v^*(s+t;\omega)}{\mu}ds} \nonumber \\
&=\frac{e^{\mu}+e^{-\mu}-2}{\mu}( t+\tau ) +\int_{-\tau}^t{\frac{a_1( \theta _s\omega )-c_1( \theta _s\omega )v^*(s;\omega)}{\mu}ds}.
\end{align}
Combining \eqref{speed1} with \eqref{speed2}, we derive $ \int_{-\tau}^t{c( s;\omega ,\mu ) ds}=\int_{-( t+\tau )}^0{c( s;\theta _t\omega ,\mu ) ds}$ for  $\tau \ge 0$ and $ t\in \mathbb{R} $.
Recall that $$ (\overline{u}^{\mu}( x,t;\omega ),\overline{v}^{\mu}( x,t;\omega )) =\min \left\{ (u^*(t;\omega),v^*(t;\omega)),(\hat{u}(x,t;\omega),\hat{u}(x,t;\omega)) \right\}$$
and $$(\hat{u}(x,t;\omega),\hat{u}(x,t;\omega))=(e^{-\mu ( x-\int_0^t{c( s ;\omega ,\mu ) ds} )},e^{-\mu ( x-\int_0^t{c( s ;\omega ,\mu ) ds} )}).$$
Then we have
\begin{align*}
\tilde{ \varPhi} &( x,t;\omega )\\
=&\lim\limits_{\tau \rightarrow \infty}u( x+\int_0^t{c( s;\omega ,\mu ) ds},t+\tau ;\overline{u}^{\mu}( \cdot ,-\tau ;\omega ) ,\overline{v}^{\mu}( \cdot ,-\tau ;\omega ) ,\theta _{-\tau}\omega )\\
=&\lim\limits_{\tau \rightarrow \infty}u( x,t+\tau ;\overline{u}^{\mu}( \cdot+\int_0^t{c( s;\omega ,\mu ) ds} ,-\tau ;\omega ) ,\overline{v}^{\mu}( \cdot+\int_0^t{c( s;\omega ,\mu ) ds} ,-\tau ;\omega ) ,\theta _{-\tau}\omega )\\
=&\lim\limits_{\tau \rightarrow \infty}u( x,t+\tau ;\overline{u}^{\mu}( \cdot ,-( t+\tau ) ;\theta _t\omega ) ,\overline{v}^{\mu}( \cdot ,-( t+\tau ) ;\theta _t\omega ) ,\theta _{-\tau}\omega )\\
=&\lim\limits_{\tau \rightarrow \infty}u( x,t+\tau ;\overline{u}^{\mu}( \cdot ,-( t+\tau ) ;\theta _t\omega ) ,\overline{v}^{\mu}( \cdot ,-( t+\tau ) ;\theta _t\omega ) ,\theta _{t-(t+\tau)}\omega )\\
=&\lim\limits_{\tau \rightarrow \infty}u( x,\tau ;\overline{u}^{\mu}( \cdot ,-\tau ;\theta _t\omega ) ,\overline{v}^{\mu}( \cdot ,-\tau ;\theta _t\omega ) ,\theta _{t-\tau}\omega )\\
=&\tilde{\varPhi} ( x,0;\theta _t\omega ).
\end{align*}
Similarly, we can get $\tilde{ \varPsi} ( x,t;\omega )=\tilde{\varPsi} ( x,0;\theta _t\omega )$, and hence $(\tilde{ \varPhi} ( x,t;\omega ),\tilde{ \varPsi} ( x,t;\omega ))=(\tilde{\varPhi}( x,0;\theta _t\omega ),$ \\ $\tilde{\varPsi}( x,0;\theta _t\omega ))$.
The claim  thus follows and we get the desired random profile $(\varPhi(x,\omega),\varPsi(x,\omega))$.

 (ii)
Let
\begin{align*}
c_*&( \omega ) \\=&\sup\{ c: \limsup_{|i|\le ct,t\to\infty} [ | u_i( t;u^0,v^0,\theta _s\omega ) -u^*(t+s;\omega)|+| v_i( t;u^0,v^0,\theta _s\omega ) -v^*(t+s;\omega)| ] \\
&\ \ \ \ \ \ \ \ \ \   =0\ \ \ \mathrm{ uniformly\ in\ } s\in\R \ \mathrm{for\ all\ }(u^0,v^0)\in l_{0}^{\infty}( \mathbb{Z} )\times l_{0}^{\infty}( \mathbb{Z} )\},
\end{align*}
where
$$
l_{0}^{\infty}( \mathbb{Z} ) =\{ u=\{ u_i \} _{i\in \mathbb{Z}}\in l^{\infty}( \mathbb{Z} ) \ :\ u_i\ge 0\ \text{for\ all\ }i\in \mathbb{Z},\ u_i=0\ \text{for\ }| i |\gg 1,\ \{ u_i \} \ne 0 \} .$$
Recall that
$$
\underline{\lambda} =\underset{t-s\rightarrow \infty}{\lim\text{inf}}\frac{1}{t-s}\int_s^t{( a_1( \theta _{\tau}\omega ) -c_1( \theta _{\tau}\omega ) v^*( \tau ;\omega ) )}d\tau,
$$
and
$$
 c_0:=\inf \limits_{\mu>0}\frac{e^{\mu}+e^{-\mu}-2+\underline{\lambda}}{\mu}.
$$

 We claim that $c_*(\omega)=c_0$ for $\omega\in\Omega_0$. In fact, we consider
 \begin{align}
 \label{one}
 \dot{u}_i(t) =u_{i+1}( t ) -2u_i( t) +u_{i-1}( t)+u_i( t)(a_1(\theta _t\omega)-c_1(\theta _t\omega)v^*(t;\omega)- b_1(\theta _t\omega)u_i( t))
 \end{align}
For any $u^0\in l^{\infty,+}(\Z)$, let $u^-(t;u^0,\omega)$ be the solution of \eqref{one} with $u^-(0;u^0,\omega)=u^0$. By comparison principle, for any $(u^0,v^0)\in l^{\infty,+}(\Z) \times l^{\infty,+}(\Z)$, we have
 \begin{align}
 \label{two}
 u_i(t;u^0,v^0,\omega)\geq u_i^-(t;u^0,\omega),\ \ \forall t\geq 0.
 \end{align}
By \cite[Remark 1.1 (1)]{CaGa19},  for any $c(\omega)$ with $0<c(\omega)<c_0$,
$$
\underset{s\in\mathbb{R},|i|\le c(\omega)t,t\rightarrow \infty}{\lim\inf}u_i^-(t;u^0,\theta_s\omega)>0.
$$
Together with \eqref{two}, we then have
$$
\underset{s\in\mathbb{R},|i|\le c(\omega)t,t\rightarrow \infty}{\lim\inf}u_i(t;u^0,v^0,\theta_s\omega)>0.
$$
Then by Lemma \ref{speed}, for any $0<c<c(\omega)$,
$$\limsup_{|i|\le ct,t\to\infty} [|u_i(t;u^0,v^0,\theta_s\omega)-u^*(t+s;\omega)|+|v_i(t;u^0,v^0,\theta_s\omega)-v^*(t+s;\omega)|]=0$$ uniformly in $s\in\R$.
which implies that $c_*(\omega)\geq c_0$.

Assume that $c_*(\omega)> c_0$ for some $\omega\in\Omega_0$.  Fix  $\gamma$, $c^{'}$ and $c^{''}$ such that
$$
c_0<\gamma< c^{'}< c^{''}<c_*(\omega).
$$
Observe that $c_0>0$. For any $(u^0,v^0)\in l_0^\infty(\Z)\times l_0^\infty(\Z)$,
\begin{align}
\label{three}
\limsup_{|i|\le c^{''}t,t\to\infty} [|u_i(t;u^0,v^0,\theta_s\omega)-u^*(t+s;\omega)|+|v_i(t;u^0,v^0,\theta_s\omega)-v^*(t+s;\omega)|]=0
\end{align}
uniformly in $s\in\R$.

Let $(u_i( t;\omega ),v_i( t;\omega )) =({\varPhi} ( i-\int_0^t{c( s;\omega) ds},\theta _t\omega),{\varPsi}( i-\int_0^t{c( s;\omega) ds},\theta _t\omega))$ be as in (i) with $\bar c_{\inf}=\gamma$. Let
$$
u^s_i=\Phi(i-[\int_0^s c(\tau;\omega)d\tau],\theta_s\omega),\ \ v^s_i=\Psi(i-[\int_0^s c(\tau;\omega)d\tau],\theta_s\omega), \quad \forall\,\, s\in\R.
$$
By (i), there is $(u^0,v^0)\in l_0^\infty(\Z)\times l_0^\infty(\Z)$ such that $$(u^0,v^0)\leq(u^s,v^s),\ \ \forall s\in\R.$$
Hence
$$
u_i(t;u^0,v^0,\theta_s\omega)\leq u_i(t;u^s,v^s,\theta_s\omega)
$$
and
$$
v_i(t;u^0,v^0,\theta_s\omega)\leq v_i(t;u^s,v^s,\theta_s\omega)$$
for $i\in\Z$, $s\in\R$ and $t\ge 0$. This together with \eqref{three} implies that
\begin{align}
\label{four}
\limsup_{|i|\le c^{''}t,t\to\infty} [|u_i(t;u^s,v^s,\theta_s\omega)-u^*(t+s;\omega)|+|v_i(t;u^s,v^s,\theta_s\omega)-v^*(t+s;\omega)|]=0
\end{align}
uniformly in $s\in\R$.

Note that
$
\int_0^{t+s}{c( \tau ;\omega ) d\tau}=\int_0^s{c( \tau ;\omega ) d\tau}+\int_0^t{c( \tau ;\theta _s\omega ) d\tau}$. By (i) again, we have
\begin{align*}
&u_i(t;u^s,v^s,\theta_s\omega)\\&=\Phi(i-\int_0^t c(\tau;\theta_s\omega)d\tau-[\int_0^s c(\tau;\omega)d\tau)],\theta_{t+s}\omega)\le \Phi(i-\int_0^{t+s} c(\tau;\omega)d\tau,\theta_{t+s}\omega),
\end{align*}
and
\begin{align*}
&v_i(t;u^s,v^s,\theta_s\omega)\\&=\Psi(i-\int_0^t c(\tau;\theta_s\omega)d\tau-[\int_0^s c(\tau;\omega)d\tau)],\theta_{t+s}\omega)\le \Psi(i-\int_0^{t+s} c(\tau;\omega)d\tau,\theta_{t+s}\omega).
\end{align*}
Then
\begin{equation}
\label{five}
\limsup_{i\ge (c^{''}-c^{'})(t+s)+\int_0^{t+s} c(\tau;\omega) d\tau,t\to\infty}[u_i(t;u^s,v^s,\theta_s\omega)+v_i(t;u^s,v^s,\theta_s\omega)]=0
\end{equation}
uniformly in $s\in\R$. It follows from \eqref{four} and \eqref{five} that
$$
\bar c_{\inf}\ge c^{'}>\gamma,
$$
which is a contradiction.
Therefore,  $c_*(\omega)=c_0$.

Suppose that $(u(t;\omega),v(t;\omega))=\{(u_i(t;\omega),v_i(t;\omega))\}_{i\in\mathbb{Z}}$ with $(u_i( t;\omega ),v_i( t;\omega )) =({\varPhi} ( i-\int_0^t{c( s;\omega) ds},\theta _t\omega),$ ${\varPsi} ( i-\int_0^t{c( s;\omega) ds},\theta _t\omega))$ is a random transition front of \eqref{main-eqn-trans1} connecting $(u^*(t;\omega),v^*(t;\omega))$ and $(0,0)$. We prove that its least mean speed $\overline{c}_{\inf}\geq c_0$.  Observe that
$\inf\limits_{x\le z}\inf\limits_{s\in \mathbb{R}}\varPhi ( x,\theta _s\omega ) >0$ and $\inf\limits_{x\le z}\inf\limits_{s\in \mathbb{R}}\varPsi ( x,\theta _s\omega ) >0$ for all $z\in \mathbb{R}$.
Therefore, we can choose $(u^0_\omega,v^0_\omega)\in l_{0}^{\infty}(\mathbb{Z})\times l_{0}^{\infty}(\mathbb{Z})$ such that
$ (u^0_\omega,v^0_\omega)\leq (\varPhi(x,\theta_s\omega),\varPsi(x,\theta_s\omega))$ for all $s\in\R$.
Let $0<\epsilon\ll 1$.
Then by $c_*( \omega )=c_0$ and the comparison principle, we have that
$$\limsup\limits_{t\rightarrow \infty}\sup\limits_{s\in \mathbb{R}} [ | u_{[ ( c_0-\epsilon ) t ]}( t;u^0_\omega,v^0_\omega,\theta _s\omega ) -u^*(t+s;\omega)|+| v_{[ ( c_0-\epsilon ) t ]}( t;u^0_\omega,v^0_\omega,\theta _s\omega ) -v^*(t+s;\omega)| ]=0,$$ and
\begin{align*}
&\liminf\limits_{t\rightarrow \infty}\inf\limits_{s\in \mathbb{R}}\{u_{[ ( c_0-\epsilon ) t ]}( t;u^0_\omega,v^0_\omega,\theta _s\omega )+v_{[ ( c_0-\epsilon ) t ]}( t;u^0_\omega,v^0_\omega,\theta _s\omega )\}\\
&\le \liminf\limits_{t\rightarrow \infty}\inf\limits_{s\in \mathbb{R}}\{u_{[ ( c_0-\epsilon ) t ]}( t;\varPhi ( \cdot ,\theta _s\omega ) ,\varPsi ( \cdot ,\theta _s\omega ) ,\theta _s\omega )+v_{[ ( c_0-\epsilon ) t ]}( t;\varPhi ( \cdot ,\theta _s\omega ) ,\varPsi ( \cdot ,\theta _s\omega ) ,\theta _s\omega )\}\\
&=\liminf\limits_{t\rightarrow \infty}\inf\limits_{s\in \mathbb{R}}\{\varPhi ( [ ( c_0-\epsilon ) t ] -\int_0^t{c( \tau ;\theta _s\omega ) d\tau ,\theta _{t+s}\omega} )+\varPsi ( [ ( c_0-\epsilon ) t ] -\int_0^t{c( \tau ;\theta _s\omega ) d\tau ,\theta _{t+s}\omega} )\}.\\
\end{align*}
Together with
$
\int_0^{t+s}{c( \tau ;\omega ) d\tau}=\int_0^s{c( \tau ;\omega ) d\tau}+\int_0^t{c( \tau ;\theta _s\omega ) d\tau},
$ we know that there is a $M(\omega)$ such that
$
  ( c_0-\epsilon ) t \leq\int_0^{t+s}{c( \tau ;\omega ) d\tau}-\int_0^s{c( \tau ;\omega ) d\tau}+M(\omega)$ for all $t>0$,  $s\in\R$. Hence,
$$\overline{c}_{\inf}=
\liminf\limits_{t\rightarrow \infty}\inf\limits_{s\in \mathbb{R}}\frac{\int_0^{t+s}{c( \tau ;\omega ) d\tau}-\int_0^s{c( \tau ;\omega ) d\tau}}{t}\geq c_0-\epsilon.
$$
By the arbitrariness of $\epsilon>0$, we get $\overline{c}_{\inf}\geq c_0$.
\end{proof}



\begin{thebibliography}{99}




\bibitem{Ba18}
\newblock X. Bao,
\newblock Transition waves for two species competition system in time heterogeneous media,
\newblock \textit{Nonlinear Anal. Real World Appl.}, \textbf{44} (2018), 128--148.


\bibitem{BaLiShWa18}
X. Bao, W.-T. Li, W. Shen and Z.-C. Wang, 
Spreading speeds and linear determinacy of time dependent diffusive cooperative/competitive systems,
\textit{J. Differential Equations}, \textbf{265} (2018), 3048--3091.

\bibitem{BaWa13}
\newblock X. Bao and Z.-C. Wang,
\newblock Existence and stability of time periodic traveling waves for a periodic bistable Lotka-Volterra
competition system,
\newblock \textit{J. Differential Equations}, \textbf{255} (2013), 2402--2435.



\bibitem{CaGa19}
\newblock F. Cao and L. Gao,
\newblock Transition fronts of KPP-type lattice random equations,
\newblock \textit{Electron. J. Differential Equations}, \textbf{2019} (2019), no. 129, 1--20.


\bibitem{CaSh17}
\newblock F. Cao and W. Shen,
\newblock Spreading speeds and transition fronts of lattice KPP equations in time heterogeneous media,
\newblock \textit{Discrete Contin. Dyn. Syst.}, \textbf{37} (2017), no. 9, 4697--4727.

\bibitem{CaoShen17}
\newblock F. Cao and W. Shen,
\newblock Stability and uniqueness of generalized traveling waves of lattice Fisher-KPP equations in heterogeneous media (in Chinese),
\newblock \textit{Sci. Sin. Math.}, \textbf{47} (2017), no. 12, 1787--1808.











\bibitem{Ch03}
S.-N. Chow, Lattice dynamical systems, in: J.W. Macki, P. Zecca (Eds.), \textit{Dynamical Systems}, in: Lecture Notes in Math.,
vol. 1822, Springer, Berlin, 2003, pp. 1--102.


\bibitem{CoGa84}
\newblock C. Conley and R. Gardner,
\newblock An application of the generalized Morse index to traveling wave solutions of a competitive reaction diffusion model,
\newblock \textit{Indiana Univ. Math. J.}, \textbf{33} (1984), 319--343.


\bibitem{Du83}
S. R. Dunbar, Traveling wave solutions of diffusive Lotka-Volterra equations, \textit{J. Math. Biol.},
\textbf{17} (1983), 11--32.

\bibitem{FaYuZh17}
J. Fang, X. Yu and X.-Q. Zhao, Traveling waves and spreading speeds for time-space periodic
monotone systems, \textit{J. Functional Analysis}, \textbf{272} (2017), 4222--4262.



\bibitem{GuHa06}
\newblock J.-S. Guo and F. Hamel,
\newblock Front propagation for discrete periodic monostable equations,
\newblock \textit{Math. Ann.}, \textbf{335} (2006), 489--525.



\bibitem{GuLi11}
J.-S. Guo and X. Liang, The minimal speed of traveling fronts for Lotka-Volterra competition
system, \textit{J. Dynam. Differential Equations}, \textbf{23} (2011), 353--363.







\bibitem{GuWu11}
\newblock J.-S. Guo and C.-H. Wu,
\newblock  Wave propagation for a two-component lattice dynamical system arising in strong competition
models,
\newblock \textit{J. Differential Equations}, \textbf{250} (2011), 3504--3533.


\bibitem{GuWu12}
\newblock J.-S. Guo and C.-H. Wu,
\newblock Traveling wave front for a two-component lattice dynamical system arising in competition models,
\newblock \textit{J. Differential Equations}, \textbf{252} (2012), 4357--4391.



\bibitem{Ho98}
Y. Hosono, The minimal spread of traveling fronts for a diffusive Lotka-Volterra competition
model, \textit{Bull. Math. Biol.}, \textbf{66} (1998), 435--448.






\bibitem{Hu10}
\newblock W. Huang,
\newblock Problem on minimum wave speed for a Lotka-Volterra reaction diffusion competition model,
\newblock \textit{J. Dynam. Differential Equations}, \textbf{22}  (2010), 285--297.





\bibitem{Ka95}
Y. Kan-on, Parameter dependence of propagation speed of traveling waves for competition-diffusion equations, \textit{SIAM J. Math. Anal.}, \textbf{26} (1995), 340--363.



\bibitem{Ka97}
Y. Kan-on, Fisher wave fronts for the Lotka-Volterra competition model with diffusion, \textit{Nonlinear Anal.}, \textbf{28} (1997), 145--164.





\bibitem{LeLiWe02}
M. Lewis, B. Li and H. Weinberger, Spreading speed and linear determinacy for two species
competition models, \textit{J. Math. Biol.}, \textbf{45} (2002), 219--233.

\bibitem{LiWeLe05}
B. Li, H. F. Weinberger and M. Lewis, Spreading speeds as slowest wave speeds for cooperative
system, \textit{Math. Biosci.}, \textbf{196} (2005), 82--98.

\bibitem{LiLiRu06}
W.-T. Li, G. Lin and S. Ruan, Existence of traveling wave solutions in delayed reaction diffusion systems with applications to diffusion competition system, \textit{Nonlinearity}, \textbf{19} (2006), 1253--1273.

\bibitem{LiZh07}
X. Liang and X.-Q. Zhao, Asymptotic speeds of spread and traveling waves for monotone
semiflows with applications, \textit{Comm. Pure Appl. Math.}, \textbf{60} (2007), 1--40.



\bibitem{LiZh10}
\newblock X. Liang and X.-Q. Zhao,
\newblock Spreading
speeds and traveling waves for abstract monostable evolution
systems,
\newblock {\it J. Functional Analysis}, \textbf{259} (2010),
857--903.



\bibitem{MaZh08} S. Ma and X.-Q.  Zhao,
\newblock Global asymptotic stability of minimal fronts in monostable lattice equations,
\newblock \textit{Discrete Contin. Dyn. Syst.}, \textbf{21} (2008), no. 1, 259--275.


\bibitem{MaPa03}
J. Mallet-Paret, Traveling waves in spatially discrete dynamical systems of diffusive type, in: J.W. Macki, P. Zecca (Eds.),
\textit{Dynamical Systems}, in: Lecture Notes in Math., vol. 1822, Springer, Berlin, 2003, pp. 231--298.


\bibitem{MiSh04}
\newblock J. Mierczy\'{n}ski and W. Shen,
\newblock Lyapunov exponents and asymptotic dynamics in random Kolmogorov models,
\newblock \textit{J. Evol. Equ.}, \textbf{4} (2004), 371--390.


\bibitem{Sh18}
\newblock R.B. Salako and W. Shen,
\newblock Long time behavior of random and nonautonomous Fisher-KPP equations. Part II. Transition fronts,
\newblock arXiv:1806.03508.


\bibitem{Sh09}
W. Shen, Spreading and generalized propagating speeds of discrete KPP models in time varying environments, \textit{Front. Math. China} \textbf{4} (2009), 523--562.





\bibitem{ShKa97}
\newblock N. Shigesada, K. Kawasaki,
\newblock \textit{Biological Invasions: Theory and Practice},
\newblock Oxford Series in Ecology and Evolution, Oxford University Press, Oxford, 1997.

\bibitem{ShSw90}
\newblock B. Shorrocks, I.R. Swingland,
\newblock \textit{Living in a Patch Environment},
\newblock Oxford
University Press, New York, 1990.

\bibitem{WeLeLi02}
H. F. Weinberger, M. Lewis and B. Li, Analysis of linear determinacy for speed in cooperative
models, \textit{J. Math. Biol.}, \textbf{45} (2002), 183--218.





\bibitem{WuZo97}
\newblock J. Wu and X. Zou,
\newblock Asymptotic and periodic boundary values problems of mixed PDEs and wave solutions of lattice differential equations,
\newblock \textit{J. Differential Equations}, \textbf{135} (1997), 315--357.



\bibitem{YuZh17}
\newblock  X. Yu and X.-Q. Zhao,
\newblock Propagation phenomena for a reaction advection diffusion competition model in a periodic habitat,
\newblock \textit{J. Dynam. Differential Equations}, \textbf{29} (2017), 41--66.



\bibitem{ZhRu11}
\newblock G. Zhao and S. Ruan,
\newblock Existence, uniqueness and asymptotic stability of time periodic traveling waves for a periodic Lotka-Volterra competition system with diffusion,
\newblock \textit{J. Math. Pures Appl.}, \textbf{95} (2011), 627--671.


\bibitem{ZiHaHu93}
\newblock B. Zinner, G. Harris and W. Hudson,
\newblock Traveling wavefronts for the discrete Fisher's equation,
\newblock \textit{J. Differential Equations}, \textbf{105} (1993), 46--62.

\end{thebibliography}
\end{document}